\documentclass[12pt]{amsart}

\usepackage{a4wide,amssymb,mathrsfs,tikz,hyperref}
\usetikzlibrary{arrows}

\newtheorem{theorem}{Theorem}
\newtheorem{lemma}{Lemma}[section]

\theoremstyle{definition}

\theoremstyle{remark}

\newtheorem{example}[lemma]{Example}

\numberwithin{equation}{section}

\newcommand{\sgn}{\mathop{\mathrm{sgn}}}

\title{Tilings for Pisot beta numeration}

\author[M.~Minervino]{Milton Minervino}
\address[M.~Minervino]{Chair of Mathematics and Statistics, Department of Mathematics and Information Technology, University of Leoben, Franz-Josef-Strasse 18, A-8700 Leoben, AUSTRIA}
\email{minervino@math.tugraz.at}
\author[W.~Steiner]{Wolfgang Steiner}
\address[W.~Steiner]{LIAFA, CNRS UMR 7089, Universit\'e Paris Diderot -- Paris 7, Case 7014, 75205 Paris Cedex 13, FRANCE}
\email{steiner@liafa.univ-paris-diderot.fr}

\thanks{The first author is supported by the Austrian Science Fund (FWF): W1230, Doctoral Program ``Discrete Mathematics''. Both authors are participants in the ANR/FWF project ``FAN -- Fractals and Numeration'' (ANR-12-IS01-0002, FWF grant I1136)}
\date{\today}

\begin{document}
\begin{abstract}
For a (non-unit) Pisot number $\beta$, several collections of tiles are associated with $\beta$-numeration.
This includes an aperiodic and a periodic one made of Rauzy fractals, a periodic one induced by the natural extension of the $\beta$-transformation and a Euclidean one made of integral beta-tiles. 
We show that all these collections (except possibly the periodic translation of the central tile) are tilings if one of them is a tiling or, equivalently, the weak finiteness property (W) holds.
We also obtain new results on rational numbers with purely periodic $\beta$-expansions; in particular, we calculate $\gamma(\beta)$ for all quadratic $\beta$ with $\beta^2 = a \beta + b$, $\gcd(a,b) = 1$. 
\end{abstract}

\maketitle

\section{Introduction}

The investigation of tilings generated by beta numeration began with the ground work of Thurston \cite{Thurston:89} who produced Euclidean tilings as geometrical picture of the expansion of numbers in a Pisot unit base~$\beta$. 
These tilings are particular instances of substitution tilings, which were introduced by Rauzy in the seminal paper \cite{Rauzy:82}. 
Since then, a~large theory for irreducible unimodular Pisot substitutions has been developed; see e.g.\ the surveys \cite{Berthe-Siegel:05,Berthe-Siegel-Thuswaldner:10}. 
Nevertheless, it is still an open question whether each irreducible unimodular Pisot substitutions naturally defines a tiling (and not a multiple covering) of the respective representation space. This question, known in this context as Pisot conjecture, is related to many different branches of mathematics, such as spectral theory, quasicrystals, discrete geometry and automata.
Note that $\beta$-substitutions can be reducible and that the Pisot conjecture does not hold for reducible substitutions; see e.g.\ \cite{Baker-Barge-Kwapisz:06}. 
However, no example of a $\beta$-substitution failing the Pisot conjecture is known.

The aim of the present paper is to study tilings associated with beta numeration in the context of Pisot numbers that are not necessarily units. 
The space where these tilings are represented consists of a suitable product of Archimedean and non-Archimedean completions of the number field~$\mathbb{Q}(\beta)$. 
The study of substitution tilings in the non-unit case started in~\cite{Siegel:03}, and further advances were achieved e.g.\ in \cite{Sing:06b,Berthe-Siegel:07,Baker-Barge-Kwapisz:06,Minervino-Thuswaldner}.
In~\cite{Akiyama-Barat-Berthe-Siegel:08}, the focus is given in particular on the connection between purely periodic $\beta$-expansions and Rauzy fractals.

In the present paper, we discuss several objects: \emph{Rauzy fractals}, \emph{natural extensions}, and \emph{integral beta-tiles}. 
We recall in Theorem~\ref{th:tprop} some of the main properties of Rauzy fractals associated with beta numeration.
It is well known that they induce an aperiodic multiple tiling of their representation space, and there are several topological, combinatorial, and arithmetical conditions that imply the tiling property. 
In the irreducible unit context, having an aperiodic tiling is equivalent to having a periodic one \cite{Ito-Rao:06}.
The situation is different when we switch to the reducible and non-unit cases. In order to have a periodic tiling, a certain algebraic hypothesis~\eqref{QM}, first introduced in \cite{Siegel-Thuswaldner:09} for substitutions, must hold, and our attention is naturally restricted to a certain \emph{stripe space} when dealing with the non-unit case. 

Another big role in the present paper is played by the natural extension of the $\beta$-shift. 
Recall that the natural extension of a (non-invertible) dynamical system is an invertible dynamical system that contains the original dynamics as a subsystem and that is minimal in a measure theoretical sense; it is unique up to metric isomorphism.
If $\beta$ is a Pisot number, then we obtain a geometric version of the natural extension of the $\beta$-shift by suspending the Rauzy fractals; see Theorem~\ref{t:natext}. This natural extension domain characterises purely periodic $\beta$-expansions \cite{Hama-Imahashi:97,Ito-Rao:05,Berthe-Siegel:07} and forms (in the unit case) a Markov partition for the associated hyperbolic toral automorphism \cite{Praggastis:99}, provided that it tiles the representation space periodically. 
The Pisot conjecture for beta numeration can be stated as follows: the natural extension of the $\beta$-shift is isomorphic to an automorphism of a compact group. 

In the non-unit case, a third kind of compact sets, studied in \cite{Berthe-Siegel-Steiner-Surer-Thuswaldner:11} in the context of shift radix systems and similar to the intersective tiles in \cite{Steiner-Thuswaldner}, turns out to be interesting. 
Integral beta-tiles are Euclidean tiles that can be seen as ``slices'' of Rauzy fractals. 
In Theorem~\ref{th:int}, we provide some of their properties.
In particular, we show that the boundary of these tiles has Lebesgue measure zero; this was conjectured in \cite[Conjecture~7.1]{Berthe-Siegel-Steiner-Surer-Thuswaldner:11}.

One of the main results of this paper is the equivalence of the tiling property for all our collections of tiles. 
We extend the results from \cite{Ito-Rao:06} to the beta numeration case (where the associated substitution need not be irreducible or unimodular), with the restriction that the quotient mapping condition~\eqref{QM} is needed for a periodic tiling with Rauzy fractals.
Our series of equivalent tiling properties also contains that for the collection of integral beta-tiles. 
We complete then our Theorem~\ref{t:tiling} by proving the equivalence of these tiling properties with the weak finiteness property~\eqref{W}, and with a spectral criterion concerning the so-called boundary graph.

Finally, we make a thorough analysis of the properties of the number-theoretical function $\gamma(\beta)$ concerning the purely periodic $\beta$-expansions.
This function was defined in \cite{Akiyama:98} and is still not well understood; see \cite{Adamczewsk-Frougny-Siegel-Steiner:10}, but note that the definition therein differs from ours for non-unit algebraic numbers.
We improve in Theorem~\ref{t:gammabeta} some results of \cite{Akiyama-Barat-Berthe-Siegel:08} and answer in Theorem~\ref{t:quadratic} some of their posed questions for quadratic Pisot numbers.

This paper is organised as follows. 
We define all our objects in Section~\ref{sec:preliminaries}.
The main results are stated in Section~\ref{sec:main-results} and illustrated by an example in Section~\ref{sec:an-example}.
Section~\ref{sec:appr-results-latt} contains some lemmas that are needed in the following.
In Section~\ref{sec:prop-rauzy-fract}, we prove the properties of Rauzy fractals and describe the natural extension.
The properties of integral beta-tiles are investigated in Section~\ref{sec:prop-integr-beta}, in particular the measure of their boundary.
In Section~\ref{sec:equiv-betw-diff}, we prove the equivalence between the different tiling properties, the weak finiteness property~\eqref{W} and the spectral radius of the boundary graph.
Finally, we analyse the gamma function in Section~\ref{sec:gamma-function} and give its explicit value for a wide class of quadratic Pisot numbers. 

We have decided to give a mostly self-contained presentation and have thus included proofs that can be found in other papers, with slight modifications. 

\section{Preliminaries} \label{sec:preliminaries}

\subsection{Beta-numeration}
Let $\beta>1$ be a real number. The map
\begin{equation} 
T:\ [0,1)\to [0,1), \quad x \mapsto \beta x - \lfloor \beta x \rfloor, 
\end{equation}
is the classical greedy \emph{$\beta$-transformation}. 
Each $x\in[0,1)$ has a \emph{(greedy) $\beta$-expansion}
\[  
x = \sum_{k=1}^\infty a_k \beta^{-k}, \quad \text{with} \quad a_k = \lfloor \beta\, T^{k-1}(x)\rfloor;  
\]
the digits $a_k$ are in $\mathcal{A} = \{ 0,1,\ldots,\lceil\beta\rceil-1\}$. 
The set of admissible sequences was characterised first by Parry \cite{Parry:60} and depends only on the limit of the expansions at~$1$.

\subsection{Representation spaces}
In all the following, let $\beta$ be a Pisot number. 
Let $K = \mathbb{Q}(\beta)$, $\mathcal{O}$~its ring of integers, and set $S=\{\mathfrak{p}:\, \mathfrak{p} \mid \infty\ \hbox{or}\ \mathfrak{p} \mid (\beta)\}$. For each (finite or infinite) prime $\mathfrak{p}$
of~$K$, we choose an absolute value $\lvert\cdot\rvert_\mathfrak{p}$ and write $K_\mathfrak{p}$ for the completion of $K$ with respect to $\lvert\cdot\rvert_\mathfrak{p}$. In all what follows, the absolute value $\lvert\cdot\rvert_\mathfrak{p}$ is chosen in the following way. Let $\xi\in K$ be given. If $\mathfrak{p}\mid\infty$, denote by $\xi^{(\mathfrak{p})}$ the associated Galois conjugate of $\xi$. If $\mathfrak{p}$ is real, we set $\lvert\xi\rvert_\mathfrak{p}=\lvert\xi^{(\mathfrak{p})}\rvert$, and if $\mathfrak{p}$ is complex, we set $\lvert\xi\rvert_\mathfrak{p}=|\xi^{(\mathfrak{p})}|^2$. Finally, if $\mathfrak{p}$ is finite, we put $\lvert\xi\rvert_\mathfrak{p}=\mathfrak{N}(\mathfrak{p})^{-v_\mathfrak{p}(\xi)}$, where $\mathfrak{N}(\cdot)$ is the norm of a (fractional) ideal and $v_\mathfrak{p}(\xi)$ denotes the exponent of $\mathfrak{p}$ in the prime ideal decomposition of the principal ideal $(\xi)$.

Define the \emph{representation space}
\[
\mathbb{K}_\beta = \prod_{\mathfrak{p} \in S} K_\mathfrak{p} = \mathbb{K}_\infty \times \mathbb{K}_\mathrm{f}, \quad \mbox{with} \quad \mathbb{K}_\infty = \prod_{\mathfrak{p}\mid\infty} K_\mathfrak{p}, \quad \mathbb{K}_\mathrm{f} = \prod_{\mathfrak{p}\mid(\beta)} K_{\mathfrak{p}}.
\]
(If $\beta$ has $r$ real and $s$ pairs of complex Galois conjugates, then $\mathbb{K}_\infty = \mathbb{R}^r \times \mathbb{C}^s$.)
We equip $\mathbb{K}_\beta$ with the product metric of the metrics defined by the absolute values~$\lvert\cdot\rvert_\mathfrak{p}$ and the product measure $\mu$ of the Haar measures~$\mu_\mathfrak{p}$, $\mathfrak{p}\in S$.
The elements of $\mathbb{Q}(\beta)$ are naturally represented in $\mathbb{K}_\beta$ by the diagonal embedding
\[
\delta:\, \mathbb{Q}(\beta) \to \mathbb{K}_\beta, \quad \xi \mapsto\prod_{\mathfrak{p}\in S} \xi.
\]
The diagonal embeddings $\delta_\infty$ and $\delta_\mathrm{f}$ are defined accordingly.

Let $\mathfrak{p}_1$ be the infinite prime satisfying $|\beta|_{\mathfrak{p}_1} = \beta$. 
Set $S' = S \setminus \{\mathfrak{p}_1\}$, and define $\mathbb{K}_\beta'$, $\mathbb{K}_\infty'$, $\delta'$, $\mu'$, etc.\ accordingly.
Let $\pi_1$ and $\pi'$ be the canonical projections from $\mathbb{K}_\beta$ to $K_{\mathfrak{p}_1}$ and~$\mathbb{K}_\beta'$, respectively.
We will also use the \emph{stripe spaces}  
\[ 
Z = \mathbb{K}_\infty \times \overline{\delta_\mathrm{f}(\mathbb{Z}[\beta])} \quad \mbox{and} \quad Z' = \mathbb{K}'_\infty \times \overline{\delta_\mathrm{f}(\mathbb{Z}[\beta])}.
\] 

\subsection{Beta-tiles}
For $x\in \mathbb{Z}[\beta^{-1}]\cap[0,1)$, define the \emph{$x$-tile} (or \emph{Rauzy fractal}) as 
\begin{equation} \label{e:xtile}
\mathcal{R}(x) = \lim_{k\rightarrow\infty} \delta'\big(\beta^k\,T^{-k}(x)\big) \subseteq \mathbb{K}'_\beta,
\end{equation}
where the limit is taken with respect to the Hausdorff distance, and let 
\begin{align*}
\mathcal{C}_\mathrm{aper} & = \big\{\mathcal{R}(x): x \in \mathbb{Z}[\beta^{-1}] \cap [0,1)\big\}, \\
\tilde{\mathcal{C}}_\mathrm{aper} & = \big\{\mathcal{R}(x): x \in \mathbb{Z}[\beta] \cap [0,1)\big\} \subseteq \mathcal{C}_\mathrm{aper}.
\end{align*}
Note that the limit in~\eqref{e:xtile} exists since $\beta^k \,T^{-k}(x)\subseteq \beta^{k+1}\,T^{-k-1}(x)$ for all $k \in \mathbb{N}$.
The sets
\[
\widehat{V} = \big\{T^k(1^-): k\geq 0\big\}, \qquad V = \big(\widehat{V} \cup\{0\}\big) \setminus \{1\},
\]
with $T^k(1^-) = \lim_{x\to1,x<1} T^k(x)$, are finite because each Pisot number is a Parry number \cite{Bertrand:77,Schmidt:80}.
For $x \in [0,1)$, let
\[
\widehat{x} = \min\big\{y \in \widehat{V}: y > x\big\}.
\]
Thus, for $v \in V$, $\widehat{v}$ is the successor of~$v$ in $V \cup \{1\}$, and $\widehat{V} = \{\widehat{v}: v \in V\}$. 
Let 
\[
L = \big\langle \widehat{V} - \widehat{V} \big\rangle_\mathbb{Z} \subseteq \mathbb{Z}[\beta]
\]
be the $\mathbb{Z}$-module generated by the differences of elements in~$\widehat{V}$ and 
\[
\mathcal{C}_\mathrm{per} = \big\{\delta'(x) + \mathcal{R}(0): x \in L\big\}.
\]
The periodic collection of tiles~$\mathcal{C}_\mathrm{per}$ is locally finite only when 
\begin{equation} \tag{QM} \label{QM}
\mathrm{rank}(L) = \deg(\beta) - 1
\end{equation}
holds, which is an analogue of the \emph{quotient mapping condition} defined in \cite{Siegel-Thuswaldner:09}; see also the definition of the anti-diagonal torus in \cite[Section~8]{Baker-Barge-Kwapisz:06}. 
A~sufficient condition for~\eqref{QM} is that $\# V = \deg(\beta)$. 
In Section~\ref{sec:periodic-tiling-with}, we give examples with $\# V > \deg(\beta)$ where \eqref{QM} holds and does not hold, respectively.

\subsection{Integral beta-tiles}
For $x \in \mathbb{Z}[\beta] \cap [0,1)$, the \emph{integral $x$-tile}
\begin{equation} \label{e:inttile}
\mathcal{S}(x) = \lim_{k\to\infty} \delta'_\infty\Big( \beta^k \big(T^{-k}(x) \cap \mathbb{Z}[\beta]\big) \Big) \subseteq \mathbb{K}'_\infty
\end{equation}
was introduced in \cite{Berthe-Siegel-Steiner-Surer-Thuswaldner:11} in the context of SRS tiles; see also~\cite{Steiner-Thuswaldner}.
Let
\[
\mathcal{C}_\mathrm{int} = \big\{ \mathcal{S}(x): x \in \mathbb{Z}[\beta] \cap [0,1) \big\}.
\]
If $\beta$ is an algebraic unit, then $\mathbb{Z}[\beta] = \mathbb{Z}[\beta^{-1}]$ and $\mathcal{S}(x) = \mathcal{R}(x)$, $\mathcal{C}_\mathrm{int} = \tilde{\mathcal{C}}_\mathrm{aper} = \mathcal{C}_\mathrm{aper}$.

\subsection{Natural extension}
We give a version of the natural extension of the $\beta$-transformation~$T$ with nice algebraic and geometric properties, in the case where $\beta$ is a Pisot number, not necessarily unit. We will do this using the $x$-tiles defined above. 
Let
\begin{gather*}
\mathscr{X} = \bigcup_{v\in V} \Big( [v,\widehat{v}) \times \big(\delta'(v) - \mathcal{R}(v)\big) \Big) \subseteq \mathbb{K}_\beta, \\ 
\mathscr{T}:\ \mathscr{X} \to \mathscr{X}, \quad \mathbf{z} \mapsto \beta\, \mathbf{z} - \delta(\lfloor \beta\, \pi_1(\mathbf{z}) \rfloor),
\end{gather*} 
\begin{align*}
\mathcal{C}_\mathrm{ext} & = \big\{ \delta(x) + \overline{\mathscr{X}}: x \in \mathbb{Z}[\beta^{-1}] \big\}, \\
\tilde{\mathcal{C}}_\mathrm{ext} & = \big\{ \delta(x) + \overline{\mathscr{X}}: x \in \mathbb{Z}[\beta] \big\}.
\end{align*}
The set~$\mathscr{X}$ is the domain and $\mathscr{T}$ the transformation of our natural extension of the beta-transformation on $[0,1)$. 
Note that one usually requires the natural extension domain to be compact. 
Here, we often prefer working with~$\mathscr{X}$ instead of its closure because it has some nice properties, e.g., it characterises the purely periodic expansions. 

\subsection{Tilings}
A~collection~$\mathcal{C}$ of compact subsets of a measurable space~$X$ is called \emph{uniformly locally finite} if there exists an integer~$k$ such that each point of~$X$ is contained in at most $k$ elements of~$\mathcal{C}$.
If moreover there exists a positive integer~$m$ such that almost every point of~$X$ is contained in exactly $m$ elements of~$\mathcal{C}$, then we call $\mathcal{C}$ a \emph{multiple tiling} of~$X$ and $m$ the \emph{covering degree} of the multiple tiling. 
If $m = 1$, then $\mathcal{C}$ is called a \emph{tiling} of~$X$.
We do not require here that $\mathcal{C}$ consists of finitely many subsets up to translation or that each element of~$\mathcal{C}$ is the closure of its interior; these additional properties hold for our collections of Rauzy fractals, but not necessarily for~$\mathcal{C}_\mathrm{int}$.

A~point of~$X$ is called \emph{exclusive point} of~$\mathcal{C}$ if it is contained in exactly one element of~$\mathcal{C}$. 
Thus, a~multiple tiling is a tiling if and only if it has an exclusive point.

\subsection{Boundary graph} \label{sec:boundary-graph-1}
The nodes of the \emph{boundary graph} are the triples $[v,x,w] \in V \times \mathbb{Z}[\beta]  \times V$ such that $x \ne 0$, $\delta'(x) \in \mathcal{R}(v) - \mathcal{R}(w) + \delta'(w-v)$, and $w-\widehat{v} < x < \widehat{w}-v$. There is an edge 
\[
[v,x,w] \stackrel{(a,b)}{\longrightarrow} [v_1,x_1,w_1]\ \mbox{if and only if}\ a,b\in\mathcal{A},\, x_1 = \tfrac{b-a+x}{\beta},\, \tfrac{a+v}{\beta} \in [v_1,\widehat{v_1}),\, \tfrac{b+w}{\beta} \in [w_1,\widehat{w_1}).
\] 
This graph provides expansions of the points that lie in two different elements of~$\mathcal{C}_\mathrm{aper}$. 
If $\mathcal{C}_\mathrm{aper}$ forms a tiling, then these points are exactly the boundary points of the tiles. 

In \cite{Akiyama-Barat-Berthe-Siegel:08}, a slightly different boundary graph is defined that determines the boundary of subtiles instead of that of Rauzy fractals. In their definition, $x$ may be in $\mathbb{Z}[\beta^{-1}]$ and it is shown to be in~$\mathcal{O}$. We will see that $\mathbb{Z}[\beta]$ is sufficient. 

\subsection{Purely periodic expansions}
Let 
\[
\mathrm{Pur}(\beta) = \big\{x \in [0,1): T^k(x) = x\ \mbox{for some}\ k \ge 1\big\}.
\]
be the set of numbers with purely periodic $\beta$-expansion.
By \cite{Bertrand:77,Schmidt:80}, we know that
\begin{equation}\label{persch}
\exists\;k\geq 0:\;T^k(x) \in \mathrm{Pur}(\beta) \quad \mbox{if and only if}\quad x \in\mathbb{Q}(\beta)\cap[0,1).
\end{equation}
Furthermore, the set $\mathrm{Pur}(\beta)$ was characterised in \cite{Hama-Imahashi:97,Ito-Rao:05,Berthe-Siegel:07} by
\begin{equation}\label{perne}
x \in \mathrm{Pur}(\beta) \quad \mbox{if and only if}\quad x \in \mathbb{Q}(\beta),\ \delta(x) \in \mathscr{X};
\end{equation}
see also \cite{Kalle-Steiner:12}.
In particular, we have
\[
\mathbb{Q} \cap \mathrm{Pur}(\beta) \subseteq \mathbb{Z}_{N(\beta)} = \big\{ p/q \in \mathbb{Q}: \gcd(q,N(\beta))=1 \big\};
\]
see e.g.\ \cite[Lemma~4.1]{Akiyama-Barat-Berthe-Siegel:08}.
Here, $N(\beta)$ denotes the norm of the algebraic number~$\beta$. 
We study the quantity
\[
\gamma(\beta) = \sup \big\{ r \in[0,1]: \mathbb{Z}_{N(\beta)} \cap [0,r) \subseteq \mathrm{Pur}(\beta) \big\}
\]
that was introduced in \cite{Akiyama:98}.

\subsection{Weak finiteness}
The arithmetical property 
\begin{equation} \tag{W} \label{W}
\forall\, x \in \mathrm{Pur}(\beta) \cap \mathbb{Z}[\beta]\ \exists\, y \in [0,1\!-\!x),\, n \in \mathbb{N}:\, T^n(x+y) = T^n(y) = 0, 
\end{equation}
turns out to be equivalent to the tiling property of our collections.

\section{Main results} \label{sec:main-results}

In the following theorem, we list some important properties of the $x$-tiles. 
Most of them can be proved exactly as in the unit case, see e.g.\ \cite{Kalle-Steiner:12}.
Some of them can also be found in \cite{Berthe-Siegel:07,Akiyama-Barat-Berthe-Siegel:08} or are direct consequences of the more general results proven in \cite{Minervino-Thuswaldner} in the substitution settings. 
For convenience, we provide a full proof in Section~\ref{sec:proof-theor-refth:tp}.

\begin{theorem} \label{th:tprop}
Let $\beta$ be a Pisot number. 
For each $x \in \mathbb{Z}[\beta^{-1}] \cap [0,1)$, the following hold:
\renewcommand{\theenumi}{\roman{enumi}}
\begin{enumerate}
\itemsep1ex
\item \label{i:t11}
$\mathcal{R}(x)$ is a non-empty compact set that is the closure of its interior.
\item \label{i:t12}
The boundary of $\mathcal{R}(x)$ has Haar measure zero.
\item \label{i:t13}
$\mathcal{R}(x) = \bigcup_{y\in T^{-1}(x)} \beta\, \mathcal{R}(y)$, and the union is disjoint in Haar measure.
\item \label{i:t14}
$\mathcal{R}(x)  - \delta'(x) \subseteq \mathcal{R}(v)  - \delta'(v)$ for all $v \in V$ with $v \le x$.
\item \label{i:t15}
$\mathcal{R}(x)  - \delta'(x) \supseteq \mathcal{R}(v)  - \delta'(v)$ for all $v \in V$ with $\widehat{v} > x$.
\end{enumerate}
Moreover, we have 
\begin{equation} \label{e:covs}
\delta\big(\mathbb{Z}[\beta^{-1}]\big) + \mathscr{X} = \mathbb{K}_\beta,\ \delta\big(\mathbb{Z}[\beta]\big) + \mathscr{X} = Z,\  \hspace{-.5em} \bigcup_{x\in\mathbb{Z}[\beta^{-1}]\cap[0,1)} \hspace{-1.5em} \mathcal{R}(x) = \mathbb{K}'_\beta,\  \hspace{-.5em} \bigcup_{x\in\mathbb{Z}[\beta]\cap[0,1)} \hspace{-1em} \mathcal{R}(x) = Z',
\end{equation}
and 
\begin{equation} \label{e:nesub}
\overline{\mathscr{X}} = \overline{\bigcup_{x\in \mathbb{Z}[\beta]\cap[0,1)} \delta(x) - \{0\}\times \mathcal{R}(x)}.
\end{equation}
\end{theorem}

The following theorem is informally stated in~\cite{Akiyama-Barat-Berthe-Siegel:08} and other papers; see \cite{Kalle-Steiner:12} for the unit case. 
Here, $B$ and $\mathscr{B}$ denote the Borel $\sigma$-algebras on $X=[0,1)$ and $\mathscr{X}$, respectively. 
The set $\mathscr{X}$ is equipped with the Haar measure~$\mu$, while $X$ is equipped with the measure $\mu \circ \pi_1^{-1}$, which is an absolutely continuous invariant measure for~$T$.

\begin{theorem} \label{t:natext}
Let $\beta$ be a Pisot number. 
The dynamical system $(\mathscr{X},\mathscr{B},\mu,\mathscr{T})$ is a natural extension of $([0,1),B,\mu\circ\pi_1^{-1},T)$. 
\end{theorem}

Some of the following properties of integral $\beta$-tiles can be found in~\cite{Berthe-Siegel-Steiner-Surer-Thuswaldner:11}; the main novelty is that we can show that the boundary has measure zero.

\begin{theorem} \label{th:int}
Let $\beta$ be a Pisot number. 
For each $x \in \mathbb{Z}[\beta] \cap [0,1)$, the following hold:
\renewcommand{\theenumi}{\roman{enumi}}
\begin{enumerate}
\itemsep1ex
\item \label{i:t31}
$\mathcal{S}(x)$ is a non-empty compact set.
\item \label{i:t32}
The boundary of $\mathcal{S}(x)$ has Lebesgue measure zero.
\item \label{i:t33}
$\mathcal{S}(x) = \bigcup_{y\in T^{-1}(x)\cap\mathbb{Z}[\beta]} \beta\, \mathcal{S}(y)$.
\item \label{i:t34}
$d_H\big(\mathcal{S}(x)-\delta_\infty'(x), \mathcal{S}(y)-\delta_\infty'(y)\big) \leq 2\, \mathrm{diam}\, \pi'_\infty(\beta^k\, \mathcal{R}(0))$ for all ${y\!\in\!(x\!+\!\beta^k \mathbb{Z}[\beta])\!\cap\![v,\widehat{v})}$, $k \in \mathbb{N}$, where $v \in V$ is chosen such that $x \in [v,\widehat{v})$, and $d_H$ denotes the Hausdorff distance with respect to some metric on~$\mathbb{K}_\infty'$. 
\item \label{i:t35}
If $\beta$ is quadratic, then $\mathcal{S}(x)$ is an interval that intersects $\bigcup_{y\in\mathbb{Z}[\beta]\cap[0,1)\setminus\{x\}} \mathcal{S}(y)$ only at its endpoints.
\end{enumerate}
Moreover, we have 
\begin{equation} \label{e:Scov}
\bigcup_{x\in\mathbb{Z}[\beta]\cap[0,1)} \mathcal{S}(x) = \mathbb{K}'_\infty,
\end{equation}
and, if $\deg(\beta) \ge 2$, 
\begin{align}
\mathcal{R}(v) & = \overline{\bigcup_{x\in\mathbb{Z}[\beta]\cap[v,\widehat{v})} \delta'(v-x) + \mathcal{S}(x) \times \delta_\mathrm{f}(\{0\})} \quad \mbox{for all} \quad v \in V, \label{e:xtileint} \\
\overline{\mathscr{X}} & = \overline{\bigcup_{x \in \mathbb{Z}[\beta] \cap [0,1)} \delta(x) - \{0\} \times \mathcal{S}(x) \times \delta_\mathrm{f}(\{0\})}. \label{e:Xint}
\end{align}
\end{theorem}

A~series of equivalent tiling conditions constitutes the core of this paper. 

\begin{theorem} \label{t:tiling}
Let $\beta$ be a Pisot number.  
Then the collections $\mathcal{C}_\mathrm{ext}$, $\tilde{\mathcal{C}}_\mathrm{ext}$, $\mathcal{C}_\mathrm{aper}$, $\tilde{\mathcal{C}}_\mathrm{aper}$, and~$\mathcal{C}_\mathrm{int}$ are multiple tilings of $\mathbb{K}_\beta$, $Z$, $\mathbb{K}'_\beta$, $Z'$, and~$\mathbb{K}'_\infty$, respectively, and they all have the same covering degree. 
The following statements are equivalent: 
\renewcommand{\theenumi}{\roman{enumi}}
\begin{enumerate}
\itemsep.5ex
\item \label{i:t41}
All collections $\mathcal{C}_\mathrm{ext}$, $\tilde{\mathcal{C}}_\mathrm{ext}$, $\mathcal{C}_\mathrm{aper}$, $\tilde{\mathcal{C}}_\mathrm{aper}$, and~$\mathcal{C}_\mathrm{int}$ are tilings.
\item \label{i:t42}
One of the collections $\mathcal{C}_\mathrm{ext}$, $\tilde{\mathcal{C}}_\mathrm{ext}$, $\mathcal{C}_\mathrm{aper}$, $\tilde{\mathcal{C}}_\mathrm{aper}$, and~$\mathcal{C}_\mathrm{int}$ is a tiling.
\item \label{i:t43}
One of the collections $\mathcal{C}_\mathrm{ext}$, $\tilde{\mathcal{C}}_\mathrm{ext}$, $\mathcal{C}_\mathrm{aper}$, $\tilde{\mathcal{C}}_\mathrm{aper}$, and~$\mathcal{C}_\mathrm{int}$ has an exclusive point.
\item \label{i:t44}
Property \eqref{W} holds.
\item \label{i:t45}
The spectral radius of the boundary graph is less than~$\beta$. 
\newcounter{enumi_saved}
\setcounter{enumi_saved}{\value{enumi}}
\end{enumerate}
If \eqref{QM} holds, then the following statement is also equivalent to the ones above:
\begin{enumerate}
\setcounter{enumi}{\value{enumi_saved}}
\item \label{i:t46}
$\mathcal{C}_\mathrm{per}$ is a tiling of~$Z'$.
\end{enumerate}
\end{theorem}

By Theorem~\ref{th:int}~(\ref{i:t35}) or e.g.\ by \cite{Akiyama-Rao-Steiner:04}, the equivalent statements of the theorem hold when $\beta$ is quadratic.

The following bound and formula for $\gamma(\beta)$ simplify those that can be found in~\cite{Akiyama-Barat-Berthe-Siegel:08}.

\begin{theorem} \label{t:gammabeta}
Let $\beta$ be a Pisot number.  
Then 
\begin{equation} \label{e:gammabeta}
\gamma(\beta) \ge \inf \bigg( \{1\} \cup \bigcup_{v\in V} \big\{x \in \mathbb{Q}\cap[v,\widehat{v}):\, \delta'_\infty(v-x) \in \pi'_\infty\big(Z'\setminus\mathcal{R}(v)\big)\big\} \bigg).
\end{equation}
If moreover $\overline{\delta_\mathrm{f}(\mathbb{Q})} = \mathbb{K}_\mathrm{f}$, then equality holds in~\eqref{e:gammabeta}.
\end{theorem}

Note that $\overline{\delta'_\infty(\mathbb{Q})}$ is a line in~$\mathbb{K}'_\infty$, thus we essentially have to determine the intersection of a line with the projection of the complement of~$\mathcal{R}(v)$. 
We are able to calculate the explicit value for $\gamma(\beta)$ for a large class of quadratic Pisot numbers. 

\begin{theorem} \label{t:quadratic}
Let $\beta$ be a quadratic Pisot number with $\beta^2 = a \beta + b$, $a \ge b \ge 1$. 
Then 
\begin{equation} \label{e:gammaquadratic}
\gamma(\beta) \ge \max\bigg\{0, 1 - \frac{(b-1)b\,\beta}{\beta^2-b^2}\bigg\},
\end{equation}
and equality holds if $\gcd(a,b) = 1$.
We have $\frac{(b-1)b\,\beta}{\beta^2-b^2} < 1$ if and only if $(b-1)b < a$.
\end{theorem}

For $a=b=2$, numerical experiments suggest that $\gamma(\beta) = 1$, while \eqref{e:gammaquadratic} only gives the trivial bound $\gamma(\beta) \ge 0$. 
Thus we believe that equality may not hold in \eqref{e:gammaquadratic} if $\gcd(a,b) > 1$.

\section{An example} \label{sec:an-example}

We illustrate our different tilings for the example $\beta=1+\sqrt{3}$, with $\beta^2 = 2\beta + 2$. 
Here, the prime $2$ ramifies in~$\mathcal{O}$, thus we get the representation space $\mathbb{K}'_\beta = \mathbb{R} \times \mathbb{K}_\mathrm{f}$, with $\mathbb{K}_\mathrm{f}\cong \mathbb{Q}_2^2$.
Each element of $\mathbb{K}_\mathrm{f}$ can be written as $\sum_{j=k}^\infty \delta_\mathrm{f}(d_j \beta^j)$, with $d_j\in\{0,1\}$, and we represent it by $\sum_{j=k}^\infty d_j 2^{-j-1}$ in our pictures.

In Figure~\ref{fig:b2aper}, a~patch of the aperiodic tiling $\mathcal{C}_\mathrm{aper}$ together with the corresponding integral beta-tiles (that form~$\mathcal{C}_\mathrm{int}$) is represented.
The aperiodic tiling $\widetilde{\mathcal{C}}_\mathrm{aper}$ constitutes the ``lowest stripe'' of~$\mathcal{C}_\mathrm{aper}$. 
Another possibility to tile the stripe~$Z'$ is given by the periodic tiling $\mathcal{C}_\mathrm{per}$ that is sketched in Figure~\ref{fig:per}.
In Figure~\ref{fig:natext}, the natural extension domain is shown, which tiles $\mathbb{K}_\beta$ and~$Z$ periodically; see Figure~\ref{f:natexttiling}.

\begin{figure}[ht]
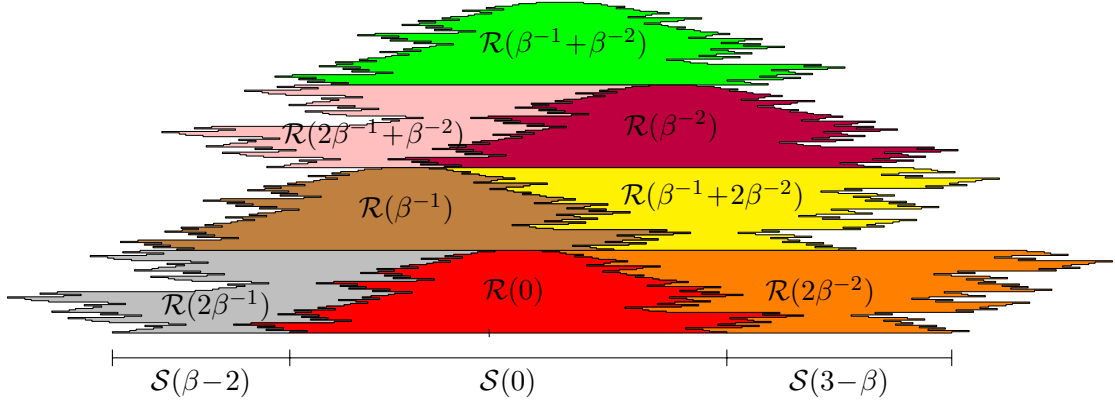

\centering

\caption{The patch $\beta^{-2}\,\mathcal{R}(0)$ of the aperiodic tiling $\mathcal{C}_\mathrm{aper}$ and the corresponding integral beta-tiles, $\beta^2 = 2\beta+2$. \label{fig:b2aper}}
\end{figure}

\begin{figure}[ht]
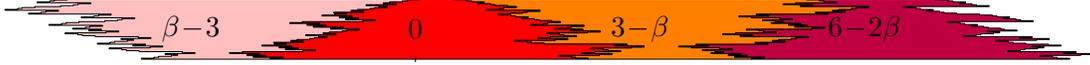

%
\caption{A patch of $\mathcal{C}_\mathrm{per}=\{\delta'(x)+\mathcal{R}(0):x \in  \mathbb{Z}\, ({\beta\!-\!3})\}$, $\beta^2=2\beta+2$.\label{fig:per}}
\end{figure}

\begin{figure}[ht]
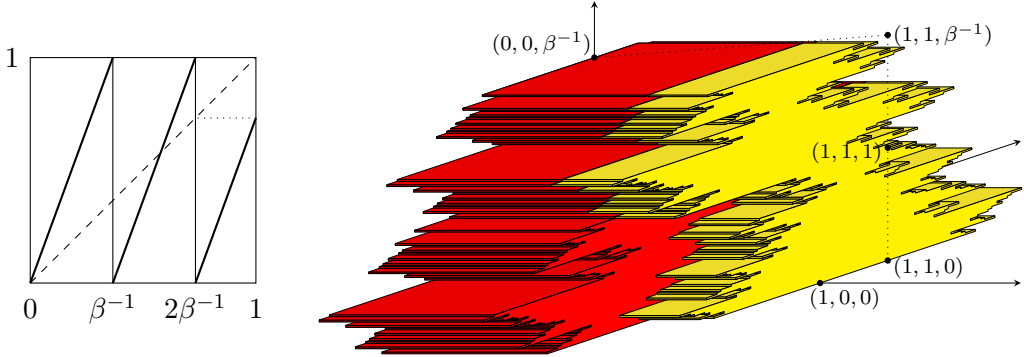

\centering
%
\caption{$\beta$-transformation and natural extension domain $\mathscr{X}$ for $\beta^2 = 2\beta+2$.\label{fig:natext}}
\end{figure} 

\begin{figure}[ht]
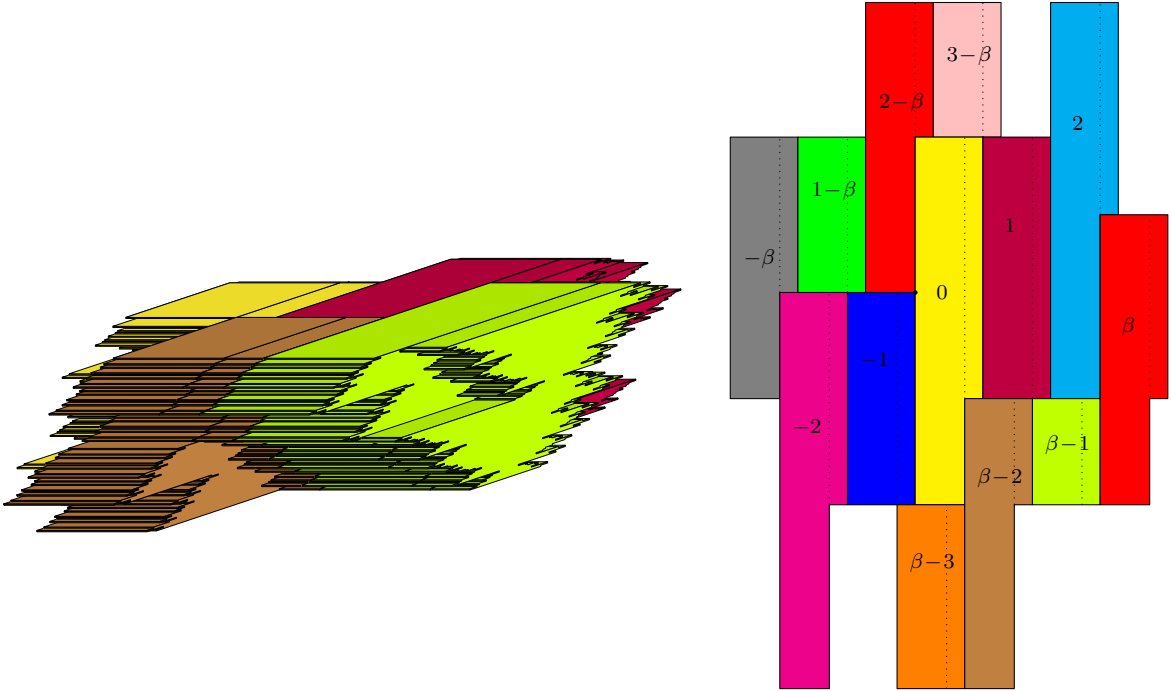

\centering
%
\caption{Periodic tiling $\tilde{\mathcal{C}}_\mathrm{ext} \subseteq \mathcal{C}_\mathrm{ext}$, $\beta^2=2\beta+2$.
On the left, the following tiles are represented: $\mathscr{X}$ (yellow), $\mathscr{X}+\delta(1)$ (purple), $\mathscr{X}+\delta(\beta-2)$ (brown), $\mathscr{X}+\delta(\beta-1)$ (light green).
On the right, the intersection with $\mathbb{K}_\infty \times \delta_\mathrm{f}(\{0\})$. \label{f:natexttiling}}
\end{figure}

\section{Approximation results and Delone sets} \label{sec:appr-results-latt}

We start with some results that are used in the proofs throughout the paper. 

\begin{lemma}[Strong Approximation Theorem, see e.g. \cite{Cassels:67}] \label{l:sat}
Let $S$ be a finite set of primes and let $\mathfrak{p}_0$ be a prime of a number field $K$ which does not belong to $S$. Let $z_\mathfrak{p}\in K_\mathfrak{p}$ be given numbers, for $\mathfrak{p}\in S$. Then, for every $\varepsilon>0$, there exists $x\in K$ such that
\[ 
|x-z_\mathfrak{p}|_\mathfrak{p} < \varepsilon\ \text{for}\ \mathfrak{p}\in S,\ \text{and}\ |x|_\mathfrak{p} \leq 1\ \text{for}\ \mathfrak{p} \notin S\cup\{\mathfrak{p}_0\}.
\]
\end{lemma}

\begin{lemma} \label{l:betaZ}
For each $x \in \mathbb{Z}[\beta^{-1}] \setminus \mathbb{Z}[\beta]$, we have $\delta_\mathrm{f}(x) \notin \overline{\delta_\mathrm{f}(\mathbb{Z}[\beta])}$.
\end{lemma}

\begin{proof}
We first show that
\begin{equation} \label{e:h}
\mathbb{Z}[\beta^{-1}] \cap \mathcal{O} \subseteq \beta^{-h} \mathbb{Z}[\beta]
\end{equation}
for some $h \in \mathbb{N}$. 
As $\mathbb{Z}[\beta]$ is a subgroup of finite index of~$\mathcal{O}$, we can choose $x_1,\ldots,x_n \in \mathcal{O}$ that form a complete set of representatives of $\mathcal{O}/\mathbb{Z}[\beta]$. 
Choose integers $h_1,\ldots, h_n$ as follows. 
If $x_i \notin \mathbb{Z}[\beta^{-1}]$, then set $h_i=0$ and notice that $(x_i+\mathbb{Z}[\beta]) \cap \mathbb{Z}[\beta^{-1}] = \emptyset$. 
If $x_i \in \mathbb{Z}[\beta^{-1}]$, then choose $h_i\geq 0$ such that $x_i \in \beta^{-h_i} \mathbb{Z}[\beta]$, hence $x_i+\mathbb{Z}[\beta] \subseteq \beta^{-h_i} \mathbb{Z}[\beta]$.
Then 
\[ 
\mathbb{Z}[\beta^{-1}] \cap \mathcal{O} = \bigcup_{i=1}^n \big(\mathbb{Z}[\beta^{-1}] \cap (\mathbb{Z}[\beta]+x_i)\big) \subseteq \beta^{-\max\{h_i\}} \mathbb{Z}[\beta],
\] 
thus~\eqref{e:h} holds with $h = \max\{h_i\}$.

Let now $x \in \mathbb{Z}[\beta^{-1}] \setminus \mathbb{Z}[\beta]$ and suppose that $\delta_\mathrm{f}(x) \in \overline{\delta_\mathrm{f}(\mathbb{Z}[\beta])}$.
Then there is $y \in \mathbb{Z}[\beta]$ such that $|y - x|_\mathfrak{p} \le |\beta^h|_\mathfrak{p}$ for all $\mathfrak{p} \mid (\beta)$, with $h$ as above, i.e., $y-x \in \beta^h\, \mathcal{O}$.
By~\eqref{e:h}, we obtain that $y-x \in \mathbb{Z}[\beta]$, contradicting that $y \in \mathbb{Z}[\beta]$ and $x \notin \mathbb{Z}[\beta]$.
\end{proof}

\begin{lemma}\label{l:qdense}
Let $(\beta)=\prod_i \mathfrak{p}_i^{m_i}$, with $\mathfrak{p}_i\mid (p_i)$.
Then $\overline{\delta_\mathrm{f}(\mathbb{Q})} = \mathbb{K}_\mathrm{f}$ if and only if $e_{\mathfrak{p}_i\mid (p_i)} = f_{\mathfrak{p}_i|(p_i)} = 1$ for all~$i$ and the prime numbers $p_i$ are all distinct.

If $\beta$ is quadratic, $\beta^2 = a \beta + b$, then $\gcd(a,b) = 1$ implies $\overline{\delta_\mathrm{f}(\mathbb{Q})} = \mathbb{K}_\mathrm{f}$. 
\end{lemma}

\begin{proof}
By Lemma~\ref{l:sat}, $\delta_\mathrm{f}(\mathbb{Q})$~is dense in $\prod_i \mathbb{Q}_{p_i}$ if and only if the $p_i$ are distinct. By $[K_{\mathfrak{p}_i}:\mathbb{Q}_{p_i}]=e_{\mathfrak{p}_i\mid (p_i)}f_{\mathfrak{p}_i\mid (p_i)}$, if either $e_{\mathfrak{p}_i\mid (p_i)}$ or $f_{\mathfrak{p}_i\mid (p_i)}$ is greater than $1$, then $\delta_\mathrm{f}(\mathbb{Q})$ cannot be dense in $\mathbb{K}_\mathrm{f}$. The other direction is similar.

If $\beta^2 = a \beta + b$ and $\gcd(a,b) = 1$, given $p\mid b$,  we have that $p\nmid \mathrm{disc}(\mathbb{Z}[\beta])=a^2+4b$. Thus $p\nmid [\mathcal{O}:\mathbb{Z}[\beta]]$, by the formula $\mathrm{disc}(\mathbb{Z}[\beta]) = [\mathcal{O}:\mathbb{Z}[\beta]]^2\cdot \mathrm{disc}(\mathbb{Q}(\beta))$ (see e.g.~\cite[Proposition~4.4.4]{Cohen:93}). Hence we can apply \cite[Theorem~4.8.13]{Cohen:93} and obtain that $(p)$ splits, since $\gcd(a,b) = 1$. This means $e_{\mathfrak{p}\mid (p)} = f_{\mathfrak{p}|(p)} = 1$ for all~$\mathfrak{p}\mid(p)$.
\end{proof}

\begin{lemma} \label{l:del}
The set $\delta(\mathbb{Z}[\beta^{-1}])$ is a lattice in $\mathbb{K}_\beta$. Furthermore each set $\delta'(\mathbb{Z}[\beta^{-1}]\cap X)$, where $X\subset\mathbb{R}$ is bounded and has non-empty interior, is a Delone set in $\mathbb{K}_\beta'$.
\end{lemma}

\begin{proof}
The first statement is a direct consequence of the fact that $\mathbb{Z}[\beta^{-1}]$ is a finite index subgroup of the ring of $S$-integers $\mathcal{O}_S=\mathcal{O}[\beta^{-1}]$, which is a lattice in $\mathbb{K}_\beta$ by the approximation theorem for number fields; see e.g.\ \cite{Weil:95}. 
For the second statement, note that $\delta'(\mathbb{Z}[\beta^{-1}]\cap X)$ is a model set.
We refer to \cite[Lemma~3.5]{Minervino-Thuswaldner} and \cite[Lemma~4.3]{Kalle-Steiner:12} for more details.
\end{proof}

\begin{lemma} \label{l:Llat}
The set $\delta(\mathbb{Z}[\beta])$ is a lattice in~$Z$.
If \eqref{QM} holds, then $\delta'(L)$ is a lattice in~$Z'$.
\end{lemma}

\begin{proof}
The sets $\delta_\infty(\mathbb{Z}[\beta])$ and, if \eqref{QM} holds, $\delta'_\infty(L)$ are Delone subgroups in $\mathbb{K}_\infty$ and~$\mathbb{K}'_\infty$, respectively.
Since $\overline{\delta_\mathrm{f}(\mathbb{Z}[\beta])}$ is compact, we obtain that $\delta(\mathbb{Z}[\beta])$ and $\delta'(L)$ are Delone subgroups in $Z$ and~$Z'$, respectively.
\end{proof}

\begin{lemma} \label{l:3Nbeta}
Assume that $\deg(\beta) \ge 2$.
For each $y \in \mathbb{Z}[\beta]$, $k \in \mathbb{N}$, $\varepsilon > 0$, we have
\[
\frac{\# \big\{x \in \big(y + \beta^k\mathbb{Z}[\beta]\big) \cap [0,1):\, \delta'_\infty(x) \in X\big\}}{\# \big\{x \in \mathbb{Z}[\beta] \cap [0,1): \delta'_\infty(x) \in X\big\}}\le \frac{2+\varepsilon}{|N(\beta)|^k}
\]
for each rectangle $X \subseteq \mathbb{K}'_\infty$ with sufficiently large side lengths.
\end{lemma}

\begin{proof}
Let $k \in \mathbb{N}$, choose a set of representatives $Y$ of $\mathbb{Z}[\beta]/\beta^k\mathbb{Z}[\beta]$ with $Y \subseteq [0,1)$, and set
\[
C(y) = \# \big\{x \in \big(y + \beta^k\mathbb{Z}[\beta]\big) \cap [0,1):\, \delta'_\infty(x) \in X\big\}.
\]
For $y, \tilde{y} \in Y$ with $y < \tilde{y}$, choose $z \in \mathbb{Z}[\beta]$ with $\tilde{y}-y \le \beta^k z \le 1$. 
(This is possible because $\beta^k \mathbb{Z}[\beta]$ is dense in~$\mathbb{R}$ by the irrationality of~$\beta$.)
Then 
\[
\big(y + \beta^k \mathbb{Z}[\beta]\big) \cap [0,1) \supseteq \Big( \big(\tilde{y} + \beta^k \mathbb{Z}[\beta]\big) \cap [\tilde{y}-y,1) + y - \tilde{y}\Big)  \cup \Big( \big(\tilde{y} + \beta^k \mathbb{Z}[\beta]\big) \cap [0,\tilde{y}-y) + \beta^k z + y  - \tilde{y}\Big),
\]
which implies the two inequalities
\begin{align*}
C(y) & \ge \# \big\{x \in \big(\tilde{y} + \beta^k\mathbb{Z}[\beta]\big) \cap [\tilde{y}-y,1):\, \delta'_\infty(x) \in X\big\} \\
& \quad - \# \big\{x \in \big(\tilde{y} + \beta^k\mathbb{Z}[\beta]\big) \cap [\tilde{y}-y,1):\, \delta'_\infty(x) \in X,\,  \delta'_\infty(x+y-\tilde{y}) \notin X\big\}, \\
C(y) & \ge \# \big\{x \in \big(\tilde{y} + \beta^k\mathbb{Z}[\beta]\big) \cap [0,\tilde{y}-y):\, \delta'_\infty(x) \in X\big\} \\
& \quad - \# \big\{x \in \big(\tilde{y} + \beta^k\mathbb{Z}[\beta]\big) \cap [0,\tilde{y}-y):\, \delta'_\infty(x) \in X,\,  \delta'_\infty(x+\beta^kz+y-\tilde{y}) \notin X\big\}.
\end{align*}
Since $\delta'_\infty(\mathbb{Z}[\beta] \cap [0,1))$ is a Delone set by Lemma~\ref{l:del}, the subtracted quantities are small compared to~$C(\tilde{y})$, provided that $(X + \delta'_\infty(y-\tilde{y})) \setminus X$  and $(X + \delta'_\infty(\beta^k z+y-\tilde{y})) \setminus X$ are small compared to~$X$.
If $X$ is a rectangle with sufficiently large side lengths, we have thus 
\[
2\, C(y) \ge \frac{C(\tilde{y})}{1+\varepsilon/2}.
\]
Similar arguments provide the same inequality for $y > \tilde{y}$, thus
\[
C(y) \ge \frac{1}{2+\varepsilon} \max_{\tilde{y}\in Y} C(\tilde{y}) 
\]
for all $y \in Y$. 
Summing over~$Y$ gives that
\[
\# \big\{x \in \mathbb{Z}[\beta] \cap [0,1):\, \delta'_\infty(x) \in X\big\} \ge \frac{|N(\beta)|^k}{2+\varepsilon} \max_{\tilde{y}\in Y} C(\tilde{y}),
\]
which proves the lemma.
\end{proof}

\section{Properties of Rauzy fractals and the natural extension}  \label{sec:prop-rauzy-fract}

\subsection{Proof of Theorem~\ref{th:tprop}} \label{sec:proof-theor-refth:tp}
Let $x \in \mathbb{Z}[\beta^{-1}] \cap [0,1)$. 
Each element of~$\mathcal{R}(x)$ is the limit of elements of $\delta'(\beta^k\,T^{-k}(x))$ and hence of the form 
\[
\mathbf{z} = \lim_{k\to\infty} \delta'\bigg(\sum_{j=0}^{k-1} a_j \beta^j + x\bigg) = \delta'(x) + \sum_{j=0}^\infty \delta'(a_j\beta^j),
\]
with $a_j \in \mathcal{A}$ and $\sum_{j=0}^{k-1} a_j \beta^j + x \in [0,\beta^k)$ for all $k \in \mathbb{N}$. 
Thus our definition of the $x$-tiles is essentially the same as in the other papers on this topic. 
In particular, $\mathcal{R}(x)$ is a compact set with uniformly bounded diameter.

\medskip
The set equation in Theorem~\ref{th:tprop}~(\ref{i:t13}) is a direct consequence of our definition.

\medskip
If $0 \le y \le x$, then we have 
\[
\beta^k\,T^{-k}(x) - x \subseteq \beta^k\,T^{-k}(y) - y
\]
for all $k \in \mathbb{N}$, which proves Theorem~\ref{th:tprop}~(\ref{i:t14}).
Let now $v \in V$ with $\widehat{v} > x$.
To see that 
\[
\beta^k\,T^{-k}(x) - x \supseteq \beta^k\,T^{-k}(v) - v
\] 
for all $k \in \mathbb{N}$, suppose that there exists $z \in (\beta^k\,T^{-k}(v) - v) \setminus (\beta^k\,T^{-k}(x) - x)$, and assume that $k$ is minimal such that this set is non-empty. 
Then we have $z + v < \beta^k \le z + x$. 
Since $T^j(\frac{z+v}{\beta^k}) \in T^{j-k}(v)$ for $1 \le j \le k$, the minimality of $k$ gives that $\beta^{k-j} T^j(\frac{z+v}{\beta^k}) + x - v \in \beta^{k-j}\,T^{j-k}(x)$, thus $\beta^{k-j} T^j(\frac{z+v}{\beta^k}) + \beta^k - z - v < \beta^{k-j}$, which implies that $T^j(1^-) = \beta^j + T^j(\frac{z+v}{\beta^k}) - \frac{z+v}{\beta^{k-j}}$. 
Hence $T^k(1^-) = \beta^k - z$ and thus $v < T^k(1^-) \le x$, which contradicts the assumption that $\widehat{v} > x$.
This proves Theorem~\ref{th:tprop}~(\ref{i:t15}).
In particular, we have that 
\begin{equation} \label{e:equal}
\mathcal{R}(x) - \delta'(x) = \mathcal{R}(v) - \delta'(v) \quad \mbox{for all} \quad x \in \mathbb{Z}[\beta^{-1}] \cap [v,\widehat{v}),\ v \in V.
\end{equation}

\medskip
We now consider the covering properties of our collections of tiles, following mainly~\cite{Kalle-Steiner:12}.
We start with a short proof of~\eqref{perne}.
Let $x \in [0,1)$ with $T^k(x) = x$. 
Then $x \in \mathbb{Q}(\beta)$ and 
\[
\delta'(0) = \lim_{n\to\infty} \delta'(\beta^{nk} x) \in \lim_{n\to\infty} \delta'(\beta^{nk} T^{-nk}(x)) \in \mathcal{R}(x),
\] 
where we have extended the definition of $\mathcal{R}(x)$ to $x \in \mathbb{Q}(\beta)$. 
If $v \in V$ is such that $x \in [v,\widehat{v})$, then \eqref{e:equal} gives that
\[
\delta(x) \in \{x\} \times (\delta'(x) - \mathcal{R}(x)) = \{x\} \times (\delta'(v) - \mathcal{R}(v)) \subseteq \mathscr{X}.
\]
On the other hand, let $x \in \mathbb{Q}(\beta)$ be such that $\delta(x) \in \mathscr{X}$.
If $q \in \mathbb{Z}$ is such that $x \in \frac{1}{q} \mathbb{Z}[\beta^{-1}]$, then we have $\mathscr{T}^{-n}(\delta(x)) \subseteq \delta(\frac{1}{q} \mathbb{Z}[\beta^{-1}])$ for all $n \in \mathbb{N}$. 
Since $\delta(\frac{1}{q} \mathbb{Z}[\beta^{-1}])$ is a lattice in~$\mathbb{K}_\beta$ by Lemma~\ref{l:del} and $\mathscr{X}$ is bounded, the set $\bigcup_{n\in\mathbb{N}} \mathscr{T}^{-n}(\delta(x))$ is finite, thus $\mathscr{T}^k(\delta(z)) = \delta(z)$ for some $k \ge 1$, $z \in \mathbb{Q}(\beta) \cap [0,1)$, with $\delta(z) \in \mathscr{T}^{-n}(\delta(x))$, $n > k$.
Since $\mathscr{T}(\delta(z)) = \delta(T(z))$, we obtain that $T^k(z) = z$ and $T^n(z) = x$, thus $T^k(x) = x$.

\medskip
Now, we use~\eqref{perne} to prove that $\delta(\mathbb{Z}[\beta^{-1}]) + \mathscr{X} = \mathbb{K}_\beta$.
We first show that 
\begin{equation} \label{e:deltaQ}
\delta(\mathbb{Q}(\beta)) \subseteq \delta(\mathbb{Z}[\beta^{-1}]) + \mathscr{X}.
\end{equation} 
Let $x \in \mathbb{Q}(\beta)$.
Then the sequence $(\beta^k x \bmod \mathbb{Z}[\beta^{-1}])_{k\in\mathbb{Z}}$ is periodic; indeed, choosing $q \in \mathbb{Z}$ such that $x \in \frac{1}{q} \mathbb{Z}[\beta^{-1}]$, we have $\beta^k x \in \frac{1}{q} \mathbb{Z}[\beta^{-1}]$ for all $k \in \mathbb{Z}$, and the periodicity of $(\beta^k x \bmod \mathbb{Z}[\beta^{-1}])_{k\in\mathbb{Z}}$ follows from the finiteness of $\frac{1}{q} \mathbb{Z}[\beta^{-1}] / \mathbb{Z}[\beta^{-1}]$.
By \eqref{persch} and~\eqref{perne}, we have $\delta(T^k(x-\lfloor x\rfloor)) \in \mathscr{X}$ for all sufficiently large $k \in \mathbb{N}$. 
Thus we can choose $k \in \mathbb{N}$ such that $\beta^k x \equiv x \mod \mathbb{Z}[\beta^{-1}]$ and $\delta(T^k(x-\lfloor x\rfloor)) \in \mathscr{X}$.
As $T^k(x-\lfloor x\rfloor) \equiv \beta^k x \mod \mathbb{Z}[\beta^{-1}]$, we obtain that $\delta(x) \in \mathscr{X} + \delta(\mathbb{Z}[\beta^{-1}])$, i.e., \eqref{e:deltaQ} holds.

Since $\overline{\mathscr{X}}$ is compact and $\delta(\mathbb{Z}[\beta^{-1}])$ is a lattice, \eqref{e:deltaQ} implies that 
\[
\delta(\mathbb{Z}[\beta^{-1}]) + \overline{\mathscr{X}} = \overline{\delta(\mathbb{Q}(\beta))} = \mathbb{K}_\beta.
\] 
Observe that $\mathscr{X}$ differs only slightly from its closure:
\begin{equation} \label{e:diffclosX}
\overline{\mathscr{X}} \setminus \mathscr{X} = \bigcup_{v\in V} \Big(\{\widehat{v}\} \times \big(\delta'(v) - \mathcal{R}(v)\big) \setminus \big(\delta'(\widehat{v}) - \mathcal{R}(\widehat{v})\big) \Big),
\end{equation}
where $\mathcal{R}(1) = \emptyset$.
As $\mathscr{X}$ is a finite union of products of a left-closed interval with a compact set, the complement of $\delta(\mathbb{Z}[\beta^{-1}]) + \mathscr{X}$ in $\mathbb{K}_\beta$ is either empty or has positive measure.
Since $\mu(\overline{\mathscr{X}} \setminus \mathscr{X}) = 0$, we obtain that
\begin{equation} \label{e:natextcov}
\delta(\mathbb{Z}[\beta^{-1}]) + \mathscr{X} = \mathbb{K}_\beta.
\end{equation}
Since $\mathcal{R}(v) -\delta(v) \subseteq \overline{\delta'(\mathbb{Z}[\beta])}$ for all $v \in V$, we have $\mathscr{X} \subseteq Z$.
By Lemma~\ref{l:betaZ}, we have thus $(\delta(x) + \mathscr{X}) \cap Z = \emptyset$ for all $x \in \mathbb{Z}[\beta^{-1}] \setminus \mathbb{Z}[\beta]$.
Together with~\eqref{e:natextcov}, this implies that 
\[
\delta(\mathbb{Z}[\beta]) + \mathscr{X} = Z.
\]

\medskip
For each $x \in \mathbb{Z}[\beta^{-1}] \cap [v,\widehat{v})$, $v \in V$, we have
\begin{equation} \label{e:Xcapx}
\mathscr{X} \cap \{x\} \times \mathbb{K}'_\beta = \{x\} \times \big(\delta'(v) - \mathcal{R}(v)\big) = \{x\} \times \big(\delta'(x) - \mathcal{R}(x)\big) = \delta(x) - \{0\}\times \mathcal{R}(x),
\end{equation}
Since $\mathbb{Z}[\beta]$ is dense in~$\mathbb{R}$, we obtain~\eqref{e:nesub}.
Rewriting~\eqref{e:Xcapx}, we get $\big(\delta(x) - \mathscr{X}) \cap \{0\} \times \mathbb{K}'_\beta = \{0\} \times \mathcal{R}(x)$, which shows together with~\eqref{e:natextcov} that
\begin{equation} \label{e:rfcov}
\bigcup_{x\in\mathbb{Z}[\beta^{-1}]\cap[0,1)} \mathcal{R}(x) = \mathbb{K}'_\beta.
\end{equation}
Since $\mathcal{R}(x) \cap Z' = \emptyset$ for all $x \in \mathbb{Z}[\beta^{-1}] \setminus \mathbb{Z}[\beta]$ by Lemma~\ref{l:betaZ}, we have
\begin{equation} \label{e:RcovZ}
\bigcup_{x\in\mathbb{Z}[\beta]\cap[0,1)} \mathcal{R}(x) = Z'.
\end{equation}

\medskip
By~\eqref{e:rfcov} and Baire's theorem, $\mathcal{R}(x)$ has non-empty interior for some $x \in \mathbb{Z}[\beta^{-1}] \cap [0,1)$. 
Using the set equations in Theorem~\ref{th:tprop}~(\ref{i:t13}) and \eqref{e:equal}, we obtain that $\mathcal{R}(x)$ has non-empty interior for all $x \in \mathbb{Z}[\beta^{-1}] \cap [0,1)$. 
Consequently, the set equations also imply that $\mathcal{R}(x)$ is the closure of its interior; see e.g.\ \cite{Kalle-Steiner:12} for more details.
This proves Theorem~\ref{th:tprop}~(\ref{i:t11}).

\medskip
For the proof of Theorem~\ref{th:tprop}~(\ref{i:t12}), we follow again \cite{Kalle-Steiner:12} and prove first that $\mathscr{T}$ is bijective up to a set of measure zero.
First note that $\mathscr{T}(\mathscr{X}) = \mathscr{X}$.
Partitioning~$\mathscr{X}$ into the sets
\[
\mathscr{X}_a = \{ \mathbf{z} \in \mathscr{X} : \lfloor \beta\, \pi_1(\mathbf{z}) \rfloor = a\} \qquad (a \in \mathcal{A}), 
\]
we have $\mathscr{T}(\mathbf{z}) = \beta\, \mathbf{z} - \delta(a)$ for all $\mathbf{z} \in \mathscr{X}_a$.
Thus $\mathscr{T}$ is injective on each~$\mathscr{X}_a$, $a \in \mathcal{A}$, and
\[  
\sum_{a\in\mathcal{A}} \mu\big(\mathscr{T}(\mathscr{X}_a)\big) = \sum_{a\in\mathcal{A}} \mu(\beta\, \mathscr{X}_a) = \sum_{a\in\mathcal{A}} \mu(\mathscr{X}_a) = \mu(\mathscr{X}) = \mu\big(\mathscr{T}(\mathscr{X})\big) = \mu\bigg( \bigcup_{a\in\mathcal{A}} \mathscr{T}(\mathscr{X}_a) \bigg),
\]
where the second equality holds by the product formula $\prod_{\mathfrak{p}\in S}|\beta|_\mathfrak{p} = 1$.
Hence $\mu(\mathscr{T}(\mathscr{X}_a) \cap \mathscr{T}(\mathscr{X}_b)) = 0$ for all $a,b\in \mathcal{A}$ with $a\neq b$, and $\mathscr{T}$ is bijective up to a set of measure zero.

For each $x \in \mathbb{Z}[\beta^{-1}] \cap [0,1)$ and sufficiently small $\varepsilon > 0$, we have 
\[
\mathscr{T}^{-1}\Big([x,x+\varepsilon) \times \big(\delta'(x) - \mathcal{R}(x)\big)\Big) = \bigcup_{y\in T^{-1}(x)} [y,y+\varepsilon\beta^{-1}) \times \big(\delta'(y) - \mathcal{R}(y)\big).
\]
Since this union is disjoint, the union in Theorem~\ref{th:tprop}~(\ref{i:t13}) is disjoint in Haar measure.
Thus if~$\mathcal{R}(y)$, $y \in T^{-k}(x)$, is in the interior of $\beta^{-k}\, \mathcal{R}(x)$, its boundary has measure zero. 
As $\mathcal{R}(x)$ has non-empty interior and multiplication by $\beta^{-1}$ is expanding on~$\mathbb{K}'_\beta$, we find for each $v \in V$ some $k \in \mathbb{N}$, $y \in T^{-k}(x) \cap [v,\widehat{v})$, such that $\mathcal{R}(y)$ is in the interior of~$\beta^{-k}\, \mathcal{R}(x)$. 
Together with~\eqref{e:equal}, this proves Theorem~\ref{th:tprop}~(\ref{i:t12}), which concludes the proof of Theorem~\ref{th:tprop}.

\subsection{Proof of Theorem~\ref{t:natext}} \label{sec:proof-theor-reft:n}
We have $\pi_1(\mathscr{X}) = [0,1)$, $T \circ \pi_1 = \pi_1 \circ \mathscr{T}$, and we know from Section~\ref{sec:proof-theor-refth:tp} that $\mathscr{T}$ is bijective on~$\mathscr{X}$ up to a set of measure zero.
It remains to show that
\[  
\bigvee_{k\in\mathbb{N}} \mathscr{T}^k \pi_1^{-1}(B) = \mathscr{B},
\]
where $\bigvee_{k\in\mathbb{N}} \mathscr{T}^k \pi_1^{-1}(B)$ is the smallest $\sigma$-algebra containing the $\sigma$-algebras $\mathscr{T}^k \pi_1^{-1}(B)$ for all $k \in \mathbb{N}$. 
It is clear that $\bigvee_{k\in\mathbb{N}} \mathscr{T}^k \pi_1^{-1}(B) \subseteq \mathscr{B}$.
For the other inclusion, we show that we can always separate in $\bigvee_{k\in\mathbb{N}} \mathscr{T}^k \pi_1^{-1}(B)$ two points $\mathbf{z}, \tilde{\mathbf{z}} \in \mathscr{X}$ with $\mathbf{z} \neq \tilde{\mathbf{z}}$.
If $\pi_1(\mathbf{z}) \neq \pi_1(\tilde{\mathbf{z}})$, then there are disjoint intervals $J, \tilde{J} \subseteq [0,1)$ with $\mathbf{z} \in \pi_1^{-1}(J)$, $\tilde{\mathbf{z}} \in \pi_1^{-1}(\tilde{J})$.
If $\pi_1(\mathbf{z}) = \pi_1(\tilde{\mathbf{z}})$, then consider the partition of $[0,1)$ into continuity intervals of~$T^k$ for large~$k$. 
For each continuity interval $J \subseteq [v,\widehat{v})$ of~$T^k$, we have $\mathscr{T}^k \pi_1^{-1}(J) = T^k(J) \times (\delta'(T^k(v)) - \beta^k\, \mathcal{R}(v))$. 
Since multiplication by~$\beta$ is contracting on $\mathbb{K}'_\beta$, we can find $k \in \mathbb{N}$ such that $\mathbf{z} \in \mathscr{T}^k \pi_1^{-1}(J)$, $\tilde{\mathbf{z}} \in \mathscr{T}^k \pi_1^{-1}(\tilde{J})$, with two disjoint intervals $J, \tilde{J} \subseteq [0,1)$.

\section{Properties of integral beta-tiles} \label{sec:prop-integr-beta}

\subsection{Basic properties}
We first prove that, for each $x \in \mathbb{Z}[\beta] \cap [0,1)$, $\mathcal{S}(x)$ is well defined and $\mathcal{S}(x) \ne \emptyset$.
To this end, we show that $(\delta'_\infty(\beta^k\, (T^{-k}(x) \cap \mathbb{Z}[\beta])))_{k\in\mathbb{N}}$ is a Cauchy sequence with respect to the Hausdorff distance~$d_H$. 
Since $\{0,1,\ldots,|N(\beta)|-1\}$ is a complete residue system of $\mathbb{Z}[\beta]/\beta\mathbb{Z}[\beta]$ and $|N(\beta)| \leq \beta$ because $\beta$ is a Pisot number, we have $T^{-1}(y) \cap \mathbb{Z}[\beta] \ne \emptyset$ for all $y \in \mathbb{Z}[\beta] \cap [0,1)$.
Using that $T^{-k-1}(x) \cap \mathbb{Z}[\beta] = T^{-1} (T^{-k}(x) \cap \mathbb{Z}[\beta]) \cap \mathbb{Z}[\beta]$, we get
\[
d_H\Big(\delta'_\infty\big(\beta^{k+1}\, (T^{-k-1}(x) \cap \mathbb{Z}[\beta])\big), \delta'_\infty\big(\beta^k\, (T^{-k}(x) \cap \mathbb{Z}[\beta])\big)\Big) \le (\lceil\beta\rceil-1)\, \|\delta'_\infty(\beta^k)\|,
\]
which tends to~$0$ exponentially fast as $k\to\infty$. 
This proves Theorem~\ref{th:int}~(\ref{i:t31}).

\medskip
Theorem~\ref{th:int}~(\ref{i:t33}) follows directly from the definition. 

\medskip 
To show Theorem~\ref{th:int}~(\ref{i:t34}), let $x,y \in \mathbb{Z}[\beta] \cap [v,\widehat{v})$ with $x-y \in \beta^k\,\mathbb{Z}[\beta]$, $v \in V$. 
We know from the proof of Theorem~\ref{th:tprop} that $\beta^k\, T^{-k}(x) - x = \beta^k\, T^{-k}(y) - y$ and thus
\[
\beta^k \big(T^{-k}(x) \cap \mathbb{Z}[\beta]\big) - x = \beta^k \Big(\big(T^{-k}(y) + \beta^{-k} (x-y)\big) \cap \mathbb{Z}[\beta]\Big) - x = \beta^k \big(T^{-k}(y) \cap \mathbb{Z}[\beta]\big) - y.
\]
Therefore, we have
\begin{multline*}
d_H\big(\mathcal{S}(x) - \delta_\infty'(x), \mathcal{S}(y) - \delta_\infty'(y)\big)  \leq
d_H\Big(\mathcal{S}(x) - \delta_\infty'(x), \delta_\infty'\big(\beta^k \big(T^{-k}(x)\cap\mathbb{Z}[\beta]\big) - x\big) \Big) \\
\hspace{6cm} + d_H\Big(\delta_\infty'\big(\beta^k \big(T^{-k}(y)\cap\mathbb{Z}[\beta]\big) - y\big), \mathcal{S}(y)-\delta_\infty'(y)\Big) \\
\leq 2 \max_{z\in\mathbb{Z}[\beta]\cap[0,1)} d_H\Big(\mathcal{S}(z), \delta_\infty'\big(\beta^k \big(T^{-k}(z) \cap \mathbb{Z}[\beta]\big)\big)\Big) \le 2\, \mathrm{diam}\, \pi'_\infty\big(\beta^k\, \mathcal{R}(0)\big).
\end{multline*}

\subsection{Slices of Rauzy fractals and~$\mathscr{X}$}
An alternative definition of the integral $x$-tile, $x \in \mathbb{Z}[\beta] \cap [0,1)$, could be
\begin{equation} \label{e:RintS}
\mathcal{R}(x) \cap \mathbb{K}'_\infty \times \delta_\mathrm{f}(\{0\}) = \mathcal{S}(x) \times \delta_\mathrm{f}(\{0\}).
\end{equation}
Indeed, the inclusion $\supseteq$ follows from $\lim_{k\to\infty} \delta_\mathrm{f}(\beta^k\, \mathbb{Z}[\beta]) = \delta_\mathrm{f}(\{0\})$.
For the other inclusion, let $\mathbf{z} \in \mathcal{R}(x) \cap \mathbb{K}'_\infty \times \delta_\mathrm{f}(\{0\})$. 
By Theorem~\ref{th:tprop}~(\ref{i:t13}), there is a sequence $(x_k)_{k\in\mathbb{N}}$ with $x_k \in T^{-k}(x)$ , $\mathbf{z} \in \beta^k\, \mathcal{R}(x_k)$ for all $k \in \mathbb{N}$. 
We have $x_k \in \mathbb{Z}[\beta]$, since otherwise we would have $\mathcal{R}(x_k) \cap Z' = \emptyset$ and thus $\beta^k\, \mathcal{R}(x_k) \cap \mathbb{K}'_\infty \times \delta_\mathrm{f}(\{0\}) = \emptyset$ by Lemma~\ref{l:betaZ}.
Therefore, we have $\mathbf{z} \in \mathcal{S}(x) \times \delta_\mathrm{f}(\{0\})$, and~\eqref{e:RintS} holds. 
Since the tiles can be obtained as the intersection with a ``hyperplane'', the equivalent of integral beta-tiles in \cite{Steiner-Thuswaldner} are called \emph{intersective tiles}.

\medskip
Now, the covering property~\eqref{e:Scov} is a direct consequence of~\eqref{e:RcovZ} and~\eqref{e:RintS}.

\medskip
To prove~\eqref{e:xtileint}, let $x \in \mathbb{Z}[\beta] \cap [v,\widehat{v})$, $v \in V$.
Then we have
\begin{multline*}
\mathcal{R}(v) \cap \mathbb{K}'_\infty \times \delta_\mathrm{f}(\{v-x\}) = \big(\mathcal{R}(x) + \delta'(v-x)\big) \cap \mathbb{K}'_\infty \times \delta_\mathrm{f}(\{v-x\}) \\
= \delta'(v-x) + \mathcal{R}(x) \cap \mathbb{K}'_\infty \times \delta_\mathrm{f}(\{0\}) = \delta'(v-x) + \mathcal{S}(x) \times \delta_\mathrm{f}(\{0\}).
\end{multline*}
Since $\delta_\mathrm{f}(\mathbb{Z}[\beta] \cap (v-\widehat{v},0])$ is dense in $\overline{\delta_\mathrm{f}(\mathbb{Z}[\beta])}$ if $\beta$ is irrational and $\mathcal{R}(x)$ is the closure of its interior, we obtain~\eqref{e:xtileint}.

\medskip
Similarly to the $x$-tiles, we can express the natural extension domain by means of integral $x$-tiles.
For each $x \in \mathbb{Z}[\beta] \cap [0,1)$, we have 
\begin{multline*}
\mathscr{X} \cap \{x\} \times \mathbb{K}'_\infty \times \delta_\mathrm{f}(\{x\}) = \{x\} \times \Big( \big(\delta'(x) - \mathcal{R}(x)\big) \cap \mathbb{K}'_\infty \times \delta_\mathrm{f}(\{x\}) \Big) \\
= \delta(x) - \{0\} \times \big( \mathcal{R}(x) \cap \mathbb{K}'_\infty \times \delta_\mathrm{f}(\{0\}) \big) = \delta(x) - \{0\} \times \mathcal{S}(x) \times \delta_\mathrm{f}(\{0\}).
\end{multline*}
By Lemma~\ref{l:sat} and since $\mathbb{Z}[\beta]$ is a subgroup of finite index of~$\mathcal{O}$, the set $\{(x,\delta_\mathrm{f}(x)): x \in \mathbb{Z}[\beta] \cap [0,1)\}$ is dense in $[0,1] \times \overline{\delta_\mathrm{f}(\mathbb{Z}[\beta])}$, provided that $\deg(\beta) \ge 2$.
As $\overline{\mathscr{X}}$ is the closure of its interior, we obtain~\eqref{e:Xint}. 

\subsection{Measure of the boundary}
In the present subsection, we show that $\mu'_\infty(\partial\mathcal{S}(x)) = 0$.
The proof is similar to that of \cite[Theorem~3~(i)]{Steiner-Thuswaldner}.

Let $x \in \mathbb{Z}[\beta] \cap [0,1)$, and $X \subseteq \mathbb{K}'_\infty$ be a rectangle containing~$\mathcal{S}(x)$. 
Since $\beta^{-n}\, \partial\mathcal{S}(x) \subseteq \bigcup_{y\in\mathbb{Z}[\beta]\cap[0,1)} \partial\mathcal{S}(y)$ by Theorem~\ref{th:int}~(\ref{i:t13}), we have
\begin{equation} \label{e:muXB}
\frac{\mu'_\infty(\partial\mathcal{S}(x))}{\mu'_\infty(X)} = \frac{\mu'_\infty(\beta^{-n}\,\partial\mathcal{S}(x))}{\mu'_\infty(\beta^{-n}X)} \le \frac{\mu'_\infty\big(\bigcup_{y\in\mathbb{Z}[\beta]\cap[0,1)} \partial\mathcal{S}(y) \cap \beta^{-n} X\big)}{\mu'_\infty(\beta^{-n} X)}
\end{equation}
for all $n \in \mathbb{N}$. 
To get an upper bound for $\mu'_\infty(\bigcup_{y\in\mathbb{Z}[\beta]\cap[0,1)} \partial\mathcal{S}(y) \cap \beta^{-n} X)$, let 
\begin{equation} \label{e:Rkv}
R_k(v) = \{z \in T^{-k}(v): \beta^k\, \mathcal{R}(z) \cap  \partial\mathcal{R}(v) \neq \emptyset\} \qquad (v \in V).
\end{equation}
Then, for each $y \in \mathbb{Z}[\beta] \cap [v,\widehat{v})$,
\begin{align*}
\partial \mathcal{S}(y) \times \delta_\mathrm{f}(\{0\}) & \subseteq \partial \mathcal{R}(y) \cap \mathbb{K}'_\infty \times \delta_\mathrm{f}(\{0\}) = \big(\partial \mathcal{R}(v) + \delta'(y-v)\big) \cap \mathbb{K}'_\infty \times \delta_\mathrm{f}(\{0\}) \\
& \subseteq \delta'(y-v) + \bigcup_{z\in R_k(v)} \beta^k\, \mathcal{R}(z) \cap \mathbb{K}'_\infty \times \delta_\mathrm{f}(\{v-y\}).
\end{align*}
If $\beta^k\, \mathcal{R}(z) \cap \mathbb{K}'_\infty \times \delta_\mathrm{f}(\{v-y\})$ is non-empty for a given $z \in R_k(v) \subseteq \beta^{-k} \mathbb{Z}[\beta]$, $v \in V$, then $y \in v - \beta^k (z + \mathbb{Z}[\beta])$ by Lemma~\ref{l:betaZ}, thus
\[
\bigcup_{y\in\mathbb{Z}[\beta]\cap[0,1)} \partial\mathcal{S}(y) \subseteq \bigcup_{v\in V} \bigcup_{z\in R_k(v)} \bigcup_{y\in (v-\beta^k z +\beta^k\mathbb{Z}[\beta])\cap[v,\widehat{v})} \big(\delta'_\infty(y-v) + \pi'_\infty(\beta^k\, \mathcal{R}(z))\big).
\]
Setting
\[
C_{k,n,v}(z) = \# \big\{ y\in (v-\beta^k z +\beta^k\mathbb{Z}[\beta])\cap[v,\widehat{v}):\, \delta'_\infty(y-v) \in \beta^{-n} X - \pi'_\infty(\beta^k\, \mathcal{R}(z)) \big\},
\]
we have that
\begin{equation} \label{e:SyRkv}
\mu'_\infty\bigg(\bigcup_{y\in\mathbb{Z}[\beta]\cap[0,1)} \partial\mathcal{S}(y) \cap \beta^{-n} X\bigg) \le \sum_{v\in V} \sum_{z\in R_k(v)} \sum_{y\in C_{k,n,v}(z)} \mu'_\infty\big(\pi'_\infty(\beta^k\, \mathcal{R}(z))\big)\big).
\end{equation}

Now, we estimate the number of terms in the sums in~\eqref{e:SyRkv}.
First, we have
\begin{equation} \label{e:est1}
\# R_k(v) = O(\alpha^k)
\end{equation}
for some $\alpha < \beta$.
Indeed, we can define, similarly to the boundary graph, a directed multiple graph with set of nodes~$V$ and $\# (T^{-1}(v) \cap [w,\widehat{w}))$ edges from $v$ to~$w$, $v,w \in V$.
Then this graph is strongly connected and the number of paths of length~$k$ starting from $v \in V$ is~$\# T^{-k}(v)$, whose order of growth is~$\beta^k$, thus the spectral radius of the graph is~$\beta$.
Since the interior of~$\mathcal{R}(0)$ is non-empty, we have $\beta^m\, \mathcal{R}(z) \cap  \partial\mathcal{R}(0) = \emptyset$ for some $z \in T^{-m}(0)$, $m \in \mathbb{N}$.
Let $p$ be the corresponding path of length~$m$ from~$0$ to~$w$, with $z \in [w,\widehat{w})$. 
Then there is $\alpha < \beta$ such that the number of paths of length~$k$ starting from $v \in V$ that avoid~$p$ is~$O(\alpha^k)$, hence \eqref{e:est1} holds.

From Lemma~\ref{l:3Nbeta}, we obtain that
\begin{equation} \label{e:est2}
\frac{C_{k,n,v}(z)}{\# \big\{x \in \mathbb{Z}[\beta] \cap [0,1):\, \delta'_\infty(x) \in \beta^{-n}\, X\big\}}\le \frac{3}{|N(\beta)|^k}
\end{equation}
for all $z \in \beta^{-k} \mathbb{Z}[\beta]$, $v \in V$, for sufficiently large~$n$.
(The subtraction of $\pi'_\infty(\beta^k\, \mathcal{R}(z))$ in the definition of~$C_{k,n,v}(z)$ is negligible when $n$ is large compared to~$k$.) 
As $\delta'_\infty(\mathbb{Z}[\beta] \cap [0,1))$ is a Delone set, we have
\begin{equation} \label{e:est3}
\# \big\{x \in \mathbb{Z}[\beta] \cap [0,1):\, \delta'_\infty(x) \in \beta^{-n}\, X\big\} = O\big(\mu'_\infty(\beta^{-n}\, X)\big).
\end{equation}

Finally, we use that
\begin{equation} \label{e:est4}
\mu'_\infty\Big(\pi'_\infty\big(\beta^k\, \mathcal{R}(z)\big)\Big) = O\bigg(\frac{|N(\beta)|^k}{\beta^k}\bigg)
\end{equation}
for all $z \in \mathbb{Z}[\beta^{-1}] \cap [0,1)$.
Inserting \eqref{e:est1}--\eqref{e:est4} in~\eqref{e:SyRkv} gives that
\[
\frac{\mu'_\infty\big(\bigcup_{y\in\mathbb{Z}[\beta]\cap[0,1)} \partial\mathcal{S}(y) \cap \beta^{-n} X\big)}{\mu'_\infty(\beta^{-n} X)} \le c\, \frac{\alpha^k}{\beta^k}
\]
for all $k \in \mathbb{N}$ and sufficiently large $n \in \mathbb{N}$, with some constant $c > 0$.
Together with~\eqref{e:muXB}, this implies that $\mu'_\infty(\partial\mathcal{S}(x)) \le c\, \frac{\alpha^k}{\beta^k}\, \mu'_\infty(X)$ for all $k \in \mathbb{N}$, i.e., $\mu'_\infty(\partial\mathcal{S}(x)) = 0$. 
This concludes the proof of Theorem~\ref{th:int}~(\ref{i:t32}).

\subsection{Quadratic Pisot numbers} \label{sec:quadr-pisot-numb}
To prove Theorem~\ref{th:int}~(\ref{i:t35}), let $\beta$ be a quadratic Pisot number, and denote by $z'$ the Galois conjugate of $z \in \mathbb{Q}(\beta)$. 
We show first that
\begin{equation} \label{e:order}
\sgn(x_1'-y_1') = \sgn(\beta')\, \sgn(x'-y')
\end{equation}
for all $x_1 \in T^{-1}(x) \cap \mathbb{Z}[\beta]$, $y_1 \in T^{-1}(y) \cap \mathbb{Z}[\beta]$, with $x, y \in \mathbb{Z}[\beta] \cap [0,1)$.
Writing $x_1 = \frac{a+x}{\beta}$, $y_1 = \frac{b+y}{\beta}$ with $a,b \in \mathcal{A}$, we have $\sgn(x_1'-y_1') = \sgn(\beta')\, \sgn(a-b+x'-y')$.
Write now $x = m\beta - \lfloor m\beta \rfloor$, $y = n \beta - \lfloor n \beta \rfloor$ with $m, n \in \mathbb{Z}$, and assume that $n = m+1$. 
Then 
\[
a-b+x'-y' = a-b + \lfloor (m+1) \beta \rfloor - \lfloor m \beta \rfloor - \beta' \ge a-b + \lfloor \beta \rfloor - \beta'.
\]
Recalling that $|\beta'| < 1$ since $\beta$ is a Pisot number, we obtain that $a-b+x'-y' > 0$ if $b < \lfloor \beta \rfloor$.
If $b = \lfloor \beta \rfloor$, then $y_1 < 1$ implies that $\lfloor \beta \rfloor + (m+1) \beta - \lfloor (m+1) \beta \rfloor < \beta$, thus $\lfloor m \beta \rfloor = \lfloor (m+1) \beta \rfloor - \lfloor \beta \rfloor - 1$, hence we also get that $\sgn(a-b+x'-y') = 1$. 
Since $\sgn(x'-y') = \sgn(n-m)$, we obtain that \eqref{e:order} holds in the case $n = m+1$, and we infer that \eqref{e:order} holds in the general case.

Inductively, we get that $\sgn(x_k' - y_k') = \sgn(\beta')^k\, \sgn(x' - y')$ for all $x_k \in T^{-k}(x) \cap \mathbb{Z}[\beta]$, $y_k \in T^{-k}(y) \cap \mathbb{Z}[\beta]$.
Since $T(\mathbb{Z}[\beta] \cap [0,1)) \subseteq  \mathbb{Z}[\beta] \cap [0,1)$, the sets $T^{-k}(x) \cap \mathbb{Z}[\beta]$, $x \in \mathbb{Z}[\beta] \cap [0,1)$, form a partition of $\mathbb{Z}[\beta] \cap [0,1)$ for each $k \in \mathbb{N}$. 
Renormalizing by $(\beta')^k$ and taking the Hausdorff limit shows that $\mathcal{S}(x)$ is an interval for each $x\in\mathbb{Z}[\beta]\cap[0,1)$ that meets $\bigcup_{y\in\mathbb{Z}[\beta]\cap[0,1)\setminus\{x\}} \mathcal{S}(y)$ only at its endpoints, i.e., $\mathcal{C}_\mathrm{int}$ is a tiling of $\mathbb{K}'_\infty = \mathbb{R}$.

\section{Equivalence between different tiling properties} \label{sec:equiv-betw-diff}

\subsection{Multiple tilings}
We first recall the proof that~$\mathcal{C}_\mathrm{aper}$ is a multiple tiling; see e.g.\ \cite{Ito-Rao:06,Kalle-Steiner:12} for the unit case.
Since $\delta'(\mathbb{Z}[\beta^{-1}]\cap[0,1))$ is a Delone set and $\mathrm{diam}\,\mathcal{R}(x)$ is uniformly bounded, $\mathcal{C}_\mathrm{aper}$~is uniformly locally finite. 
Suppose that $\mathcal{C}_\mathrm{aper}$ is not a multiple tiling.
Then there are integers $m_1 > m_2$ such that some set~$U$ of positive measure is covered at least $m_1$ times by elements of~$\mathcal{C}_\mathrm{aper}$ and some point~$\mathbf{z}$ lies in exactly $m_2$ tiles.
Since the set $\bigcup_{x\in\mathbb{Z}[\beta^{-1}]\cap[0,1)} \partial \mathcal{R}(x)$ has measure zero, we can assume that $U$ is contained in the complement of this set, and we can choose $U$ to be open. 
For small $\varepsilon>0$, the configuration of the tiles in a large neighbourhood of~$\delta'(x)$ does not depend on $x \in \mathbb{Z}[\beta^{-1}] \cap [0,\varepsilon)$.
(In the parlance of \cite{Ito-Rao:06,Kalle-Steiner:12}, the collection~$\mathcal{C}_\mathrm{aper}$ is quasi-periodic.)
Therefore, each element of $\mathbf{z} + \delta'(\mathbb{Z}[\beta^{-1}] \cap [0,\varepsilon))$ lies in exactly $m_2$ tiles.
By the set equations in Theorem~\ref{th:tprop}~(\ref{i:t13}), $\beta^{-k} U$ is covered at least $m_1$ times for all $k \in \mathbb{N}$.
Since $\delta'(\mathbb{Z}[\beta^{-1}] \cap [0,\varepsilon))$ is a Delone set and $\beta^{-k} U$ is arbitarily large, we have a contradiction.

\medskip
As $\tilde{\mathcal{C}}_\mathrm{aper}$ is the restriction of~$\mathcal{C}_\mathrm{aper}$ to~$Z'$, it is a multiple tiling with same covering degree. 

\medskip
The multiple tiling property of~$\mathcal{C}_\mathrm{ext}$ also follows from that of~$\mathcal{C}_\mathrm{aper}$.
Indeed, for $x,y \in \mathbb{Z}[\beta^{-1}]$, we have $(\delta(x) + \mathscr{X}) \cap \{y\} \times \mathbb{K}'_\beta \ne \emptyset$ if and only if $y-x \in [0,1)$.
Since
\[
\big(\delta(x) + \mathscr{X}\big) \cap \{y\} \times \mathbb{K}'_\beta = \{y\} \times \big(\delta'(y) - \mathcal{R}(y-x)) = \delta(y) - \{0\} \times \mathcal{R}(y-x)
\]
if $y-x \in [0,1)$ and $\mathbb{Z}[\beta^{-1}]$ is dense in~$\mathbb{R}$, we obtain that $\mathcal{C}_\mathrm{ext}$ is a multiple tiling with same covering degree as~$\mathcal{C}_\mathrm{aper}$.
Again, $\tilde{\mathcal{C}}_\mathrm{ext}$~being the restriction of~$\mathcal{C}_\mathrm{ext}$ to~$Z$, it is a multiple tiling with same covering degree. 

\medskip
Almost the same proof as for~$\mathcal{C}_\mathrm{aper}$ shows that~$\mathcal{C}_\mathrm{int}$ is a multiple tiling. 
The collection~$\mathcal{C}_\mathrm{int}$ need not be quasi-periodic, but we have ``almost quasi-periodicity'' by Theorem~\ref{th:int}~(\ref{i:t34}). 
If $\mathbf{z} \in \mathbb{K}'_\infty$ lies in exactly $m_2$ tiles, then for large $k \in \mathbb{N}$ each element of $\mathbf{z} + \delta'_\infty(\beta^k \mathbb{Z}[\beta] \cap [0,\varepsilon))$ lies in at most $m_2$ tiles.
As $\delta'_\infty(\beta^k \mathbb{Z}[\beta] \cap [0,\varepsilon))$ is Delone, we get that $\mathcal{C}_\mathrm{int}$ is a multiple~tiling.

\medskip
The relation between $\tilde{\mathcal{C}}_\mathrm{aper}$ and~$\mathcal{C}_\mathrm{int}$ is similar to that between $\mathcal{C}_\mathrm{ext}$ and~$\mathcal{C}_\mathrm{aper}$, as one is obtained from the other by intersection with a suitable ``hyperplane''. 
However, the proof that $\mathcal{C}_\mathrm{int}$ has the same covering degree as~$\tilde{\mathcal{C}}_\mathrm{aper}$ needs a bit more attention than that for~$\mathcal{C}_\mathrm{ext}$.
Let $m$ be the covering degree of~$\tilde{\mathcal{C}}_\mathrm{aper}$ and choose $y_1,\ldots,y_m \in \mathbb{Z}[\beta] \cap [0,1)$ such that the interior~$U$ of $\bigcap_{i=1}^m \mathcal{R}(y_i)$ is non-empty; then $U \cap \mathcal{R}(x) = \emptyset$ for all $x \notin \{y_1,\ldots,y_m\}$. 
Set 
\[
\varepsilon = \min \big\{\widehat{x} - x:\, x \in \mathbb{Z}[\beta] \cap (-1,1),\ \delta'_\infty(x) \in \pi'_\infty\big(U-\mathcal{R}(0)\big) \big\},
\]
with $\widehat{x} = 0$ for $x < 0$.
As $\delta_\mathrm{f}(\mathbb{Z}[\beta] \cap [0,\varepsilon))$ is dense in  $\overline{\delta_\mathrm{f}(\mathbb{Z}[\beta])}$, there exists $z \in \mathbb{Z}[\beta] \cap [0,\varepsilon)$ such that $U \cap \mathbb{K}'_\infty \times \delta_\mathrm{f}(\{-z\}) \ne \emptyset$.
Then the set
\[
\tilde{U} = \delta'_\infty(z) + \pi'_\infty\big(U \cap \mathbb{K}'_\infty \times \delta_\mathrm{f}(\{-z\})\big)
\] 
is open.
If $\mathcal{S}(x+z) \cap\tilde{U} \neq \emptyset$ for $x \in \mathbb{Z}[\beta] \cap [-z,1-z)$, then we get from $\mathcal{S}(x+z) \subseteq \delta'_\infty(x+z) + \pi'_\infty(\mathcal{R}(0))$ that $\delta'_\infty(x+z) \in \tilde{U} - \pi'_\infty(\mathcal{R}(0))$, i.e., $\delta'_\infty(x) \in \pi'_\infty(U-\mathcal{R}(0))$.
Therefore, we have $\widehat{x} - x \geq \varepsilon$, which ensures that $x \in [0,1-z)$ because $x \in [-z,0)$ would mean that $\widehat{x} - x = -x \le z < \varepsilon$.
Now, $x + z < x + \varepsilon \le \widehat{x}$ implies that
\[
\mathcal{R}(x) \cap \mathbb{K}'_\infty \times \delta_\mathrm{f}(\{-z\}) = \big(\mathcal{R}(x+z) - \delta'(z)\big) \cap \mathbb{K}'_\infty \times \delta_\mathrm{f}(\{-z\}) = \mathcal{S}(x+z) \times \delta_\mathrm{f}(\{0\}) - \delta'(z).
\]
From $U \subseteq \mathcal{R}(y_i)$, we obtain that $\pi'_\infty(U \cap \mathbb{K}'_\infty \times \delta_\mathrm{f}(\{-z\})) \subseteq \mathcal{S}(y_i+z) - \delta'_\infty(z)$, thus 
\[
\tilde{U} \subseteq \mathcal{S}(y_i+z)  \quad \mbox{for} \quad 1 \le i \le m.
\]
We have already seen that $\mathcal{S}(x+z) \cap \tilde{U} = \emptyset$ for $x \in \mathbb{Z}[\beta] \cap [-z,0)$.
For $x \in \mathbb{Z}[\beta] \cap {[0,1\!-\!z)} \setminus \{y_1,\ldots,y_m\}$, the disjointness of $\mathcal{S}(x+z)$ and~$\tilde{U}$ follows from $U \cap \mathcal{R}(x) = \emptyset$. 
Thus, $\tilde{U}$~is a set of positive measure that is covered exactly $m$ times, i.e., the covering degree of $\mathcal{C}_\mathrm{int}$~is~$m$.

\medskip
We have proved that the collections $\mathcal{C}_\mathrm{ext}$, $\tilde{\mathcal{C}}_\mathrm{ext}$, $\mathcal{C}_\mathrm{aper}$, $\tilde{\mathcal{C}}_\mathrm{aper}$, and~$\mathcal{C}_\mathrm{int}$ are all multiple tilings with the same covering degree. 
The equivalence of Theorem~\ref{t:tiling} (\ref{i:t41})--(\ref{i:t43}) follows immediately.

\subsection{Property \eqref{W}} \label{sec:property-w}
Several slightly different but equivalent definitions of weak finiteness can be found in \cite{Hollander:96,Akiyama:02,Sidorov:03,Akiyama-Rao-Steiner:04}.
Our definition of~\eqref{W}, which is called~(H) in~\cite{Akiyama-Rao-Steiner:04}, is essentially due to Hollander. 
By~\cite{Akiyama-Rao-Steiner:04}, this property holds for each quadratic Pisot number, for each cubic Pisot unit, as well as for each $\beta > 1$ satisfying $\beta^d = t_1 \beta^{d-1} + t_2 \beta^{d-2} + \cdots + t_{d-1} \beta + t_d$ for some $t_1, \ldots,t_d \in \mathbb{Z}$ with $t_1 > \sum_{k=2}^d|t_k|$.
Other classes of numbers giving tilings and thus satisfying \eqref{W} were found by \cite{Barge-Kwapisz:05,Baker-Barge-Kwapisz:06}.

\medskip
An immediate consequence of~\eqref{perne} is that, for $x \in \mathbb{Z}[\beta^{-1}] \cap [0,1)$, 
\[
\delta'(0) \in\mathcal{R}(x) \quad \mbox{if and only if} \quad  x \in P = \mathrm{Pur}(\beta) \cap \mathbb{Z}[\beta].
\]
(Note that we cannot have $T^n(x) = x$ for $x \in \mathbb{Z}[\beta^{-1}] \setminus \mathbb{Z}[\beta]$.)
For general $z \in \mathbb{Z}[\beta^{-1}] \cap [0,\infty)$, let $n \in \mathbb{N}$ be such that $y + \beta^{-n} z < \widehat{y}$ for all $y \in P$.
Then we have
\begin{equation} \label{e:zinRx}
\delta'(z) \in \mathcal{R}(x) \quad \mbox{if and only if} \quad  x \in T^n(P + \beta^{-n}z).
\end{equation}
To prove~\eqref{e:zinRx}, let first $x = T^n(y + \beta^{-n}z)$, $y \in P$, and let $k \ge 1$ be such that $T^k(y) = y$. 
Then $y + \beta^{-n} z < \widehat{y}$ implies that $T^{n+jk}(y + \beta^{-n-jk}z) = T^n(y + \beta^{-n}z) = x$ for all $j \in \mathbb{N}$; cf.\ the proof of Theorem~\ref{th:tprop}~(\ref{i:t15}) in Section~\ref{sec:proof-theor-refth:tp}.
Thus we have $\delta'(\beta^{n+jk} (y + \beta^{-n-jk}z)) \in \mathcal{R}(x)$, i.e., $\delta'(z) \in \mathcal{R}(x) - \delta'(\beta^{n+jk}y)$. 
Taking the limit for $j \to \infty$, we conclude that $\delta'(z) \in \mathcal{R}(x)$.

Let now $\delta'(z) \in \mathcal{R}(x)$, $x \in \mathbb{Z}[\beta^{-1}] \cap [0,1)$, $z \in \mathbb{Z}[\beta^{-1}] \cap [0,\infty)$.
For each $k \in \mathbb{N}$, there is an $x_k \in T^{-k}(x)$ such that $\delta'(\beta^{-k} z) \in \mathcal{R}(x_k)$.
Then the set $\{\delta'(\beta^{-k}z - x_k):\, k\in\mathbb{N}\}$ is bounded and thus finite by Lemma~\ref{l:del}. 
Choose $y \in \mathbb{Z}[\beta^{-1}]$ such that $x_k - \beta^{-k}z = y$ for infinitely many $k \in \mathbb{N}$. 
Let $j$ and~$k$ be two successive elements of the set $\{k \in \mathbb{N}:\, x_k - \beta^{-k}z = y\}$, with $j$ large enough such that $y + \beta^{-j}z < \widehat{y}$. 
Then $T^{k-j}(y + \beta^{-k}z) = y + \beta^{-j}z$ and $T^{k-j}(y + \beta^{-k}z) = T^{k-j}(y) + \beta^{-j}z$ thus
$T^{k-j}(y) =  y$, which implies that $y \in P$ and $T^j(y + \beta^{-j}z) = T^j(x_j) = x$.
We infer that $x \in T^n(P + \beta^{-n}z)$ for all $n \in \mathbb{N}$ such that $y + \beta^{-n} z < \widehat{y}$ for all $y \in P$, thus~\eqref{e:zinRx} holds.

\medskip
We use~\eqref{e:zinRx} to prove the equivalence between \eqref{W} and the tiling property of~$\mathcal{C}_\mathrm{aper}$, similarly to \cite{Akiyama:02,Kalle-Steiner:12}.
If $\mathcal{C}_\mathrm{aper}$ is a tiling, then $\mathcal{R}(0)$ contains an exclusive point.
Since $\delta'(\mathbb{Z}[\beta^{-1}] \cap [0,\infty))$ is dense in~$\mathbb{K}'_\beta$, we have thus an exclusive point $\delta'(z) \in \mathcal{R}(0)$ with $z \in \mathbb{Z}[\beta^{-1}] \cap [0,\infty)$.
By~\eqref{e:zinRx}, this implies that $T^n(P + \beta^{-n}z) = \{0\}$ for all $n \in \mathbb{N}$ satisfying $x + \beta^{-n} z < \widehat{x}$ for all $x \in P$, in particular $T^n(\beta^{-n}z) = 0$ because $0 \in P$.
Hence, \eqref{W} holds with $y = \beta^{-n} z$ (independently from $x \in P$).

Now, we assume that \eqref{W} holds and construct an exclusive point $\delta'(z) \in \mathcal{R}(0)$ of~$\mathcal{C}_\mathrm{aper}$, with $z \in \mathbb{Z}[\beta^{-1}] \cap [0,\infty)$, as in \cite{Akiyama:02,Kalle-Steiner:12}.
Note first that \eqref{W} implies that
\begin{equation} \label{e:W'}
\forall\, x \in P,\, \varepsilon > 0\ \exists\, y \in [0,\varepsilon),\, n \in \mathbb{N}:\, T^n(x+y) = T^n(y) = 0.
\end{equation}
Indeed, if $x = \max\{T^j(x):\, j \in \mathbb{N}\}$, $T^k(x) = x$, and $T^n(x+y) = T^n(y) = 0$, then we have $T^{\ell k-j}(T^j(x) + \beta^{j-\ell k} y) = x+y$ and thus $T^{n+\ell k-j}(T^j(x) + \beta^{j-\ell k}y) = 0 = T^{n+\ell k-j}(\beta^{j-\ell k}y)$ for all $\ell \ge 1$, $0 \le j < k$, which proves~\eqref{e:W'}; see also \cite[Lemma~1]{Akiyama-Rao-Steiner:04}.

If $P = \{0\}$, then $\delta'(0)$ is an exclusive point. 
Otherwise, let $P = \{0,x_1,\ldots,x_h\}$ and set
\[
z = \beta^{k_1+k_2+\cdots+k_h} y_1 + \beta^{k_2+\cdots+k_h} y_2 + \cdots + \beta^{k_h} y_h,
\]
where $y_j$ and $k_j$ are defined recursively for $1 \le j \le h$ with the following properties:
\begin{itemize}
\itemsep1ex
\item
$T^{k_j}\big(T^{k_1+\cdots+k_{j-1}}(x_j + \sum_{i=1}^{j-1} y_i\,\beta^{-k_1-\cdots-k_{i-1}}) + y_j \big) = T^{k_j}(y_j) = 0$,
\item
$T^{k_1+\cdots+k_j}\big(x_{j+1} + \sum_{i=1}^j y_i\,\beta^{-k_1-\cdots-k_{i-1}}\big) \in P$, if $j < h$,
\item
$x + \sum_{i=\ell}^j y_i\,\beta^{-k_\ell-\cdots-k_{i-1}} < \widehat{x}$ for all $x \in P$, $1 \le \ell \le j$.
\end{itemize}
Then we have
\begin{align*}
& T^{k_1+\cdots+k_h}(x_j + \beta^{-k_1-\cdots-k_h}z) = T^{k_1+\cdots+k_h}\big(x_j + \textstyle{\sum_{i=1}^h} y_i\, \beta^{-k_1-\cdots-k_{i-1}}\big) \\
& = T^{k_j+\cdots+k_h}\Big( T^{k_1+\cdots+k_{j-1}}\big(x_j + \textstyle{\sum_{i=1}^{j-1}} y_i\, \beta^{-k_1-\cdots-k_{i-1}}\big) + y_j + \textstyle{\sum_{i=j+1}^h} y_i\, \beta^{-k_j-\cdots-k_{i-1}} \Big) \\
& = T^{k_{j+1}+\cdots+k_h}\big(\textstyle{\sum_{i=j+1}^h} y_i\, \beta^{-k_{j+1}-\cdots-k_{i-1}}\big) = T^{k_{j+2}+\cdots+k_h}\big(\textstyle{\sum_{i=j+2}^h} y_i\, \beta^{-k_{j+2}-\cdots-k_{i-1}}\big) = \cdots = 0
\end{align*}
for $0 \le j \le h$, with $x_0 = 0$.
By~\eqref{e:zinRx}, $\delta'(z) $ is an exclusive point of~$\mathcal{C}_\mathrm{aper}$ (in~$\mathcal{R}(0)$).

Therefore, we have proved that Theorem~\ref{t:tiling}~(\ref{i:t44}) is equivalent to (\ref{i:t41})--(\ref{i:t43}).

\subsection{Boundary graph} \label{sec:boundary-graph}
We study properties of the boundary graph defined in Section~\ref{sec:boundary-graph-1} and prove the tiling condition Theorem~\ref{t:tiling}~(\ref{i:t45}).
First note that the boundary graph is finite since, for each node $[v,x,w]$, $\delta'_\infty(x)$~is contained in the intersection of the Delone set $\delta'_\infty(\mathbb{Z}[\beta] \cap (-1,1))$ with the bounded set $\pi'_\infty(\mathcal{R}(0) - \mathcal{R}(0))$.

\medskip
Next, we show that the labels of the infinite paths in the boundary graph provide pairs of expansions exactly for the points that lie in $\mathcal{R}(x) \cap \mathcal{R}(y)$ for some $x,y \in \mathbb{Z}[\beta^{-1}] \cap [0,1)$, $x \neq y$.
When $\mathcal{C}_\mathrm{aper}$ is a tiling, these points are exactly the boundary points of~$\mathcal{R}(x)$, $x \in \mathbb{Z}[\beta^{-1}] \cap [0,1)$, hence the name ``boundary graph''. 
(In general, ``intersection graph'' might be a better name.) 
If $x \in \mathbb{Z}[\beta^{-1}] \cap [v,\widehat{v})$, $y \in \mathbb{Z}[\beta^{-1}] \cap [w,\widehat{w})$, $v,w \in V$, $x \ne y$, then 
\begin{equation} \label{e:xyintersection}
\mathcal{R}(x) \cap \mathcal{R}(y) \ne \emptyset \quad \mbox{if and only if} \quad \mbox{$[v,y\!-\!x,w]$ is a node of the boundary graph},
\end{equation}
and 
\begin{equation} \label{e:xyintersection2} 
\mathbf{z} \in \mathcal{R}(x) \cap \mathcal{R}(y) \quad \mbox{if and only if} \quad \mathbf{z} = \delta'(x) + \sum_{k=0}^\infty \delta'(a_k \beta^k) = \delta'(y) + \sum_{k=0}^\infty \delta'(b_k \beta^k)
\end{equation} 
for the sequence of labels $(a_0,b_0) (a_1,b_1) \cdots$ of an infinite path starting in ${[v,y\!-\!x,w]}$.

For $x \in \mathbb{Z}[\beta^{-1}] \cap [v,\widehat{v})$, $y \in \mathbb{Z}[\beta^{-1}] \cap [w,\widehat{w})$, we have $\mathcal{R}(x) \cap \mathcal{R}(y) \ne \emptyset$ if and only if 
\[
\delta'(0) \in \mathcal{R}(x) - \mathcal{R}(y) = \mathcal{R}(v) - \mathcal{R}(w) + \delta'(w-v) - \delta'(y-x),
\]
i.e., $\delta'(y-x) \in \mathcal{R}(v) - \mathcal{R}(w) + \delta'(w-v) \subseteq \overline{\delta'(\mathbb{Z}[\beta])}$ and thus $y-x \in \mathbb{Z}[\beta]$ by Lemma~\ref{l:betaZ}.
If moreover $x \neq y$, this is equivalent to $[v,{y\!-\!x},w]$ being a node of the boundary graph, which proves~\eqref{e:xyintersection}.
Let $\mathbf{z} \in \mathcal{R}(x) \cap \mathcal{R}(y)$.
Using Theorem~\ref{th:tprop}~(\ref{i:t13}), we have
\[
\mathcal{R}(x)\cap\mathcal{R}(y) = \bigcup_{\substack{x_1\in T^{-1}(x) \\ y_1\in T^{-1}(y)}}
\beta\, \mathcal{R}(x_1)\cap \beta\, \mathcal{R}(y_1) = \hspace{-5em} \bigcup_{\textstyle[v,y\!-\!x,w]\!\stackrel{\scriptstyle(a_0,b_0)}\to\![v_1,\frac{\scriptstyle b_0-a_0+y-x}{\scriptstyle\beta},w_1]} \hspace{-5em} \beta\, \Big( \mathcal{R}\big(\tfrac{a_0+x}{\beta}\big) \cap \mathcal{R}\big(\tfrac{b_0+y}{\beta}\big) \Big),
\]
where the transitions are edges in the boundary graph, with $\frac{a_0+x}{\beta} \in [v_1, \widehat{v_1})$, $\frac{b_0+y}{\beta} \in [w_1, \widehat{w}_1)$.
Thus we have $\mathbf{z} \in \beta\, (\mathcal{R}(x_1) \cap \mathcal{R}(y_1))$ for some $x_1 = \frac{a_0+x}{\beta} \in T^{-1}(x)$, $y_1 = \frac{b_0+y}{\beta} \in T^{-1}(y)$, and ${[v,y\!-\!x,w]} \stackrel{(a_0,b_0)}\to {[v_1,y_1\!-\!x_1,w_1]}$ is an edge in the boundary graph.
Iterating this observation, we get a path in the boundary graph labelled by $(a_0,b_0) (a_1,b_1) \cdots$ such that, for each $k \in \mathbb{N}$, $\mathbf{z} \in \beta^k\, (\mathcal{R}(x_k) \cap \mathcal{R}(y_k))$ with $x_k = \big(\sum_{j=0}^{k-1} a_j \beta^j + x\big)\, \beta^{-k} \in T^{-k}(x)$, $y_k = \big(\sum_{j=0}^{k-1} b_j \beta^j + y\big)\, \beta^{-k} \in T^{-k}(y)$.
Since $\lim_{k\to\infty} \beta^k\, \mathcal{R}(x_k) = \delta'(x) + \sum_{j=0}^\infty \delta'(a_j \beta^j)$ and $\lim_{k\to\infty} \beta^k\, \mathcal{R}(y_k) = \delta'(y) + \sum_{j=0}^\infty \delta'(b_j \beta^j)$, we obtain one direction of~\eqref{e:xyintersection2}. 
The other direction of~\eqref{e:xyintersection2} is proved similarly.

\medskip
Next, we prove that $\mathcal{C}_\mathrm{aper}$ is a tiling if and only if $\varrho < \beta$, where $\varrho$ denotes the spectral radius of the boundary graph. 
We have seen that, for $x \in \mathbb{Z}[\beta^{-1}] \cap [v,\widehat{v})$, $y \in \mathbb{Z}[\beta^{-1}] \cap [w,\widehat{w})$ with $\mathcal{R}(x) \cap \mathcal{R}(y) \ne 0$, each path of length~$k$ starting from the node ${[v,y\!-\!x,w]}$ in the boundary graph corresponds to a pair $x_k \in T^{-k}(x)$, $y_k \in T^{-k}(y)$ with $\mathcal{R}(x_k) \cap \mathcal{R}(y_k) \ne \emptyset$. 
As, for each $\tilde{x} \in \mathbb{Z}[\beta^{-1}] \cap [0,1)$, there is a bounded number of $\tilde{y} \in \mathbb{Z}[\beta^{-1}] \cap [0,1)$ such that $\mathcal{R}(x)\cap\mathcal{R}(y) \ne \emptyset$, the number of these paths gives, up to a multiplicative constant, the number of subtiles $\beta^k\, \mathcal{R}(x_k)$ of~$\mathcal{R}(x)$ that meet $\bigcup_{y\in\mathbb{Z}[\beta^{-1}]\cap[0,1)\setminus\{x\}} \mathcal{R}(y)$.
Hence we have $\mu'\big(\mathcal{R}(x) \cap \bigcup_{y\in\mathbb{Z}[\beta^{-1}]\cap[0,1)\setminus\{x\}} \mathcal{R}(y)\big) \le P(k)\, \varrho^k \beta^{-k}$
for all $k \in \mathbb{N}$, with some polynomial~$P(k)$.
If $\varrho < \beta$, this yields that $\mathcal{C}_\mathrm{aper}$ is a tiling.
On the other hand, if $\mathcal{C}_\mathrm{aper}$ is a tiling, then $\partial\mathcal{R}(x) = \mathcal{R}(x)\cap \bigcup_{y\in\mathbb{Z}[\beta^{-1}]\cap[0,1)\setminus\{x\}} \mathcal{R}(y)$, thus the number of paths of length $k$ from ${[v,y\!-\!x,w]}$ is bounded by a constant times~$R_k(v)$, with $R_k(v)$ as in~\eqref{e:Rkv}.
Since $\# R_k(v) = O(\alpha^k)$ with $\alpha<\beta$ by~\eqref{e:est1}, we have that $\varrho \le \alpha < \beta$.
Therefore, Theorem~\ref{t:tiling}~(\ref{i:t45}) is equivalent to (\ref{i:t41})--(\ref{i:t43}).

\medskip
The equivalence between Theorem~\ref{t:tiling} (\ref{i:t41})--(\ref{i:t42}) and~(\ref{i:t45}) can also be extended to multiple tilings. 
To this end, one defines a~generalisation of the boundary graph that recognises all points that lie in $m$ tiles (instead of $2$ tiles). 
Then the spectral radius of this graph is less than~$\beta$ if and only if the covering degree of $\mathcal{C}_\mathrm{aper}$ is less than~$m$.

\subsection{Periodic tiling with Rauzy fractals} \label{sec:periodic-tiling-with}
The last of our equivalent tiling conditions is that of the periodic tiling, under the condition that \eqref{QM} holds. 
This condition is satisfied when the size of~$V$ is equal to the degree of the algebraic number~$\beta$; the following examples show that \eqref{QM} can be true or false when $\# V > \deg(\beta)$. 

\begin{example} \label{ex:QM}
Let $\beta > 1$ satisfy $\beta^3 = t\beta^2 - \beta + 1$ for some integer $t\geq 2$.
Then 
\[
T(1^-) = \beta - (t-1) = \tfrac{(t-1)\beta^2+1}{\beta^3},\ T^2(1^-) = \tfrac{(t-1)\beta^2+1}{\beta^2} - (t-1) = \tfrac{1}{\beta^2},\ T^3(1^-) = \tfrac{1}{\beta},\ T^4(1^-) = 1,
\]
thus $\widehat{V} =\{1, \beta-(t-1), \beta^2-(t-1)\beta-(t-1), \beta^2-t\beta+1\}$, and $L = \langle \beta-t, \beta^2-t\beta \rangle_\mathbb{Z}$.
Therefore, \eqref{QM} holds. 
\end{example}

\begin{example}
Let $\beta>1$ satisfy $\beta^3 = t\beta^2 + (t+1) \beta + 1$ for some integer $t \geq 0$. (For $t = 0$, $\beta$ is the smallest Pisot number.)
Then
\[
T(1^-) = \beta - (t+1) = \tfrac{t\beta+1}{\beta^4}, T^2(1^-) = \tfrac{t\beta+1}{\beta^3}, T^3(1^-) = \tfrac{t\beta+1}{\beta^2}, T^4(1^-) = \tfrac{t\beta+1}{\beta} - t = \tfrac{1}{\beta}, T^5(1^-) = 1,
\]
thus $\widehat{V} = \{1, \beta-(t+1), \beta^2-(t+1)\beta, -\beta^2+(t+1)\beta+1, \beta^2-t\beta-(t+1)\}$. 
Since 
\[
1 = \big(1-\beta^2+(t+1)\beta\big) + \big(\beta^2-(t+1)\beta\big) = \big(1 - T^2(1^-)\big) + \big(1 - T^3(1^-)\big) \in L,
\]
we have $\widehat{V} \subseteq L$ and thus $L = \mathbb{Z}[\beta]$, hence \eqref{QM} does not hold. 
According to \cite{Ei-Ito:05} (see also \cite{Ei-Ito-Rao:06}), the central tile $\mathcal{R}(0)$ associated with the smallest Pisot number $\beta$ cannot tile periodically its representation space $\mathbb{K}'_\beta = \mathbb{C}$.
\end{example}

The central tile~$\mathcal{R}(0)$ is closely related to the set of non-negative \emph{$\beta$-integers}
\[
\mathbb{N}_\beta = \bigcup_{n\geq 0}\beta^n\, T^{-n}(0),
\]
as $\mathcal{R}(0) = \overline{\delta'(\mathbb{N}_\beta)}$.
We know from \cite{Thurston:89,Fabre:95,Frougny-Gazeau-Krejcar:03} that the sequence of distances between consecutive elements of~$\mathbb{N}_\beta$ is the fixed point of the \emph{$\beta$-substitution}~$\sigma$, which can be defined on the alphabet~$\widehat{V}$ by
\[
\sigma(x) = \underbrace{1\,1\,\cdots\,1}_ {\lceil T(x^-)\rceil-1\,\text{times}} T(x^-) \qquad (x \in \widehat{V}). 
\]
More precisely, we have $\mathbb{N}_\beta = \big\{ \sum_{k=0}^{n-1} w_k: n \in \mathbb{N}\big\}$, where $w_0 w_1 \cdots \in \widehat{V}^\mathbb{N}$ is the infinite word starting with $\sigma^k(1)$ for all $k \in \mathbb{N}$.
Similarly to \cite[Proposition~3.4]{Ito-Rao:06}, we obtain that
\begin{equation} \label{e:dec}
L + \mathbb{N}_\beta = \bigg\{\sum_{v\in V} c_v\, \widehat{v}: c_v \in \mathbb{Z},\, \sum_{v\in V} c_v \ge 0\bigg\}, 
\end{equation}
using that $L = \{\sum_{v\in V} c_v\, \widehat{v}:\, c_v \in \mathbb{Z},\, \sum_{v\in V} c_v = 0\}$, which implies that 
\[
L + \sum_{k=0}^{n-1} w_k = \bigg\{\sum_{v\in V} c_v\, \widehat{v}:\, c_v \in \mathbb{Z},\, \sum_{v\in V} c_v = n\bigg\} 
\]
for all $n \in \mathbb{N}$.
Next, we prove that
\begin{equation} \label{e:percov}
\delta'(L) + \mathcal{R}(0) = Z'.
\end{equation}
If \eqref{QM} does not hold, then $\delta'(L)$ is dense in $\overline{\delta'(\mathbb{Z}[\beta])} = Z'$, hence \eqref{e:percov} follows from the fact that $\mathcal{R}(0)$ has non-empty interior.
If \eqref{QM} holds, then it is sufficient to prove that
\begin{equation} \label{e:LNbeta}
\overline{\delta'(L + \mathbb{N}_\beta)} = \overline{\delta'(\mathbb{Z}[\beta])},
\end{equation}
as $\delta'(L)$ is a lattice in~$Z'$ by Lemma~\ref{l:Llat} and $\mathcal{R}(0)$ is compact.
Since $\widehat{V}$ spans~$\mathbb{Z}[\beta]$, we can write each $x \in \mathbb{Z}[\beta]$ as $x = \sum_{v\in V} c_v\, \widehat{v}$, with $c_v \in \mathbb{Z}$.
By \eqref{QM}, we have $\beta^k \notin L$ for infinitely many $k \in \mathbb{N}$, thus $x + (\sum_{v\in V} c_v) \beta^k \in L + \mathbb{N}_\beta$ or $x - (\sum_{v\in V} c_v) \beta^k \in L + \mathbb{N}_\beta$ for these~$k$. 
Since $\lim_{k\to\infty} \delta'\big(x \pm (\sum_{v\in V} c_v) \beta^k\big) = \delta'(x)$, we obtain that~\eqref{e:LNbeta} and thus \eqref{e:percov} holds.

\medskip
Throughout the rest of the subsection, assume that \eqref{QM} holds.
Then we have 
\[
\mathbb{N}_\beta \cap L = \{0\}
\]
because $\sum_{k=0}^{n-1} w_k \in L$ for some $n \ge 1$ implies that $\{\sum_{v\in V} c_v\, \widehat{v}:\, c_v \in \mathbb{Z},\, \sum_{v\in V} c_v = n\} \subseteq L$, in particular $n\, \widehat{v} \in L$ for all $v \in V$, contradicting~\eqref{QM}.
We immediately obtain that
\begin{equation} \label{e:Nbetapartitition} 
\mathbb{N}_\beta \cap (x + \mathbb{N}_\beta) = \emptyset \quad \mbox{for all} \quad x \in L \setminus \{0\},
\end{equation}
i.e., $\{x + \mathbb{N}_\beta:\, x \in L\}$ forms a partition of $L + \mathbb{N}_\beta$. 
It is natural to expect from this partition property that $\mathcal{C}_\mathrm{per} = \{\delta'(x) + \overline{\delta'(\mathbb{N}_\beta)}:\, x \in L\}$ is a tiling, but this may not be true due to the effects of taking the closure. 
We can only prove that the tiling property of $\mathcal{C}_\mathrm{per}$ is equivalent to that of~$\mathcal{C}_\mathrm{aper}$, similarly to \cite[Proposition 3.5]{Ito-Rao:06} and \cite[Proposition~6.72~(v)]{Sing:06b}.

Suppose that $\mathcal{C}_\mathrm{aper}$ is a tiling.
From~\eqref{e:percov}, we know that $\mathcal{C}_\mathrm{per}$ covers~$Z'$.
Consider $\mathcal{R}(0) \cap (\delta'(x) + \mathcal{R}(0))$ for some $x \in L \setminus \{0\}$, and assume w.l.o.g.\ that $x > 0$. 
As~$\beta^{-k}\, \mathcal{R}(0) = \bigcup_{y\in T^{-k}(0)} \mathcal{R}(y)$, we have that
\[
\beta^{-k}\, \Big(\mathcal{R}(0) \cap \big(\delta'(x) + \mathcal{R}(0)\big) \Big) = \bigcup_{y,z\in T^{-k}(0)} \Big( \mathcal{R}(y) \cap \big(\delta'(\beta^{-k} x) + \mathcal{R}(z)\big) \Big).
\]
If $z + \beta^{-k} x < \widehat{z}$, then $\delta'(\beta^{-k} x) + \mathcal{R}(z) = \mathcal{R}(z + \beta^{-k} x)$.
Since $y \neq z + \beta^{-k} x$ for all $y, z \in T^{-k}(0)$ by~\eqref{e:Nbetapartitition}, the tiling property of~$\mathcal{C}_\mathrm{aper}$ implies that $\mu'(\mathcal{R}(y) \cap \mathcal{R}(z + \beta^{-k} x)) = 0$ (if $z  + \beta^{-k} x < 1$). 
Therefore, only tiles $\delta'(\beta^{-k} x) + \mathcal{R}(z)$ with $z + \beta^{-k} x \ge \widehat{z}$ can contribute to the measure of $\mathcal{R}(0) \cap (\delta'(x) + \mathcal{R}(0))$.
The inequality $z + \beta^{-k} x \ge \widehat{z}$ is equivalent to $z \in [\widehat{v} - \beta^{-k} x, \widehat{v})$ for some $v \in V$.
As $\beta^k\, T^{-k}(0) \subseteq \mathbb{N}_\beta$  and the distances between consecutive elements in~$\mathbb{N}_\beta$ are in~$\widehat{V}$, there are at most $(\#V)\,x \,/ \min \widehat{V}$ numbers $z \in T^{-k}(0)$ satisfying this inequality.
This implies that
\[
\mu'\Big( \mathcal{R}(0) \cap \big(\delta'(x) + \mathcal{R}(0)\big) \Big) \le \frac{(\#V)\,x}{\min\widehat{V}}\, \mu'\big(\beta^k\, \mathcal{R}(0)\big) = \frac{(\#V)\,x}{\min\widehat{V}}\, \mu'\big(\mathcal{R}(0)\big)\, \beta^{-k}
\]
for all $k \in \mathbb{N}$, hence $\mathcal{R}(0) \cap (\delta'(x) + \mathcal{R}(0))$ has measure zero, thus $\mathcal{C}_\mathrm{per}$~is a tiling.

If $\mathcal{C}_\mathrm{aper}$ is not a tiling, then it covers~$\mathbb{K}'_\beta$ at least twice, and a similar proof as for the tiling property shows that $\mathcal{R}(0) \cap \bigcup_{x\in L\setminus\{0\}} (\delta'(x) + \mathcal{R}(0))$ has positive measure.
Hence, we have proved that Theorem~\ref{t:tiling}~(\ref{i:t46}) is equivalent to~(\ref{i:t41}), which concludes the proof of  Theorem~\ref{t:tiling}.

\medskip
With some more effort, the proof above can be adapted to show that $\mathcal{C}_\mathrm{per}$ is a multiple tiling with same covering degree as~$\mathcal{C}_\mathrm{aper}$.

\medskip
A~patch of the periodic tiling induced by $\beta^3=2\beta^2-\beta+1$, which satisfies~\eqref{QM} by Example~\ref{ex:QM}, is depicted in Figure~\ref{fig:redper}.

\begin{figure}[ht]
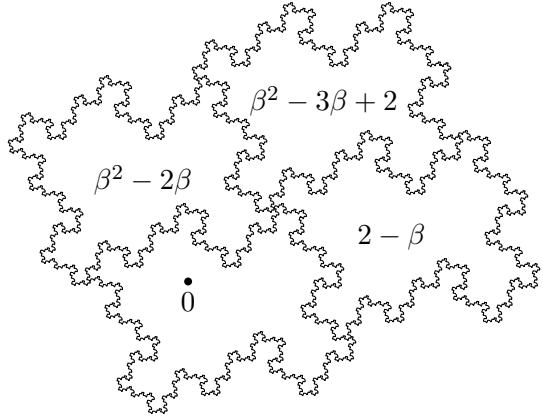


\caption{Patch of the periodic tiling $\mathcal{R}(0)+\delta'(\langle \beta-2,\beta^2-2\beta\rangle_\mathbb{Z})$ induced by $\beta^3=2\beta^2-\beta+1$.\label{fig:redper}}
\end{figure}

\section{Gamma function} \label{sec:gamma-function}

\subsection{Proof of Theorem~\ref{t:gammabeta}}
To prove Theorem~\ref{t:gammabeta}, suppose first that \eqref{e:gammabeta} does not hold.
Then there exists $y \in \mathbb{Z}_{N(\beta)} \cap [0,1)$ with $\delta(y) \notin \mathscr{X}$ and 
\begin{equation} \label{e:ymin}
y < \inf \bigg( \{1\} \cup \bigcup_{v\in V} \big\{x \in \mathbb{Q}\cap[v,\widehat{v}):\, \delta'_\infty(v-x) \in \pi'_\infty\big(Z'\setminus\mathcal{R}(v)\big)\big\} \bigg).
\end{equation}
Let $v \in V$ be such that $y \in [v,\widehat{v})$.
Then $\delta'(y) \notin \delta'(v) - \mathcal{R}(v)$ and thus $\delta'(v-y) \in Z' \setminus \mathcal{R}(v)$ because $\delta'(\mathbb{Z}_{N(\beta)}\cup V) \subseteq Z'$.
This contradicts~\eqref{e:ymin}, hence~\eqref{e:gammabeta} holds.

Assume now that $\overline{\delta_\mathrm{f}(\mathbb{Q})} = \mathbb{K}_\mathrm{f}$.
We want to prove the opposite inequality of~\eqref{e:gammabeta}. 
Since $\gamma(\beta) \le 1$, the inequality is clearly true if $\{x \in \mathbb{Q}\cap[v,\widehat{v}):\, \delta'_\infty(v-x) \in \pi'_\infty(Z'\setminus\mathcal{R}(v))\} = \emptyset$ for all $v \in V$.
Otherwise, we show that arbitrarily close to each $x \in \mathbb{Q}\cap[v,\widehat{v})$ with $\delta'_\infty(v-x) \in \pi'_\infty(Z'\setminus\mathcal{R}(v))$, we can find $y \in \mathbb{Z}_{N(\beta)}$ with $\delta(y) \notin \mathscr{X}$.
Indeed, for sufficiently small $\varepsilon > 0$, we have 
\[
\delta'_\infty(v-y) \in \pi'_\infty(Z'\setminus\mathcal{R}(v)) \quad \mbox{for all} \quad y \in \mathbb{Q}\cap(x,x+\varepsilon)
\]
because $\pi'_\infty(Z'\setminus\mathcal{R}(v))$ is open. 
By \cite[Lemma~4.7]{Akiyama-Barat-Berthe-Siegel:08}, we have 
\[
\overline{\delta\big(\mathbb{Z}_{N(\beta)}\cap(x,x+\varepsilon)\big)} = \overline{\delta_\infty\big(\mathbb{Q}\cap(x,x+\varepsilon)\big)} \times \overline{\delta_\mathrm{f}\big(\mathbb{Z}_{N(\beta)}\big)}.
\]
The set 
\[
\{\mathbf{z} \in Z'\setminus\mathcal{R}(v):\, \pi'_\infty(\mathbf{z}) \in  \mathbb{Q}\cap(x,x+\varepsilon)\} \subseteq \delta'_\infty\big(\mathbb{Q}\cap(x,x+\varepsilon)\big) \times \overline{\delta_\mathrm{f}(\mathbb{Z}[\beta])}
\]
is non-empty and open in $\delta'_\infty(\mathbb{Q}) \times \mathbb{K}_\mathrm{f}$.
Since $\overline{\delta_\mathrm{f}(\mathbb{Q})} = \mathbb{K}_\mathrm{f}$ implies that $\overline{\delta_\mathrm{f}(\mathbb{Z}_{N(\beta)})} = \overline{\delta_\mathrm{f}(\mathbb{Z}[\beta])}$, we obtain that this set contains some $\delta'(y)$ with $y \in \mathbb{Z}_{N(\beta)} \cap(x,x+\varepsilon)$.
Then we have $\delta(y) \notin \mathscr{X}$.
Since $\varepsilon$ can be chosen arbitrary small, this concludes the proof of Theorem~\ref{t:gammabeta}.

\subsection{Boundary in the quadratic case}
Let now $\beta$ be a quadratic Pisot number. 
Then we show that the boundary of $\mathcal{R}(x)$ is simply the intersection with two of its neighbours.
More precisely, for each $x \in \mathbb{Z}[\beta^{-1}] \cap [0,1)$, we have that
\begin{equation} \label{e:quadboundary} 
\partial \mathcal{R}(x) = \mathcal{R}(x) \cap \Big( \mathcal{R}\big(x + \beta - \lfloor x + \beta \rfloor\big) \cup \mathcal{R}\big(x - \beta - \lfloor x - \beta \rfloor\big) \Big).
\end{equation}
There may be other neighbours of~$\mathcal{R}(x)$, but they meet~$\mathcal{R}(x)$ only in points that also lie in $\mathcal{R}(x \pm \beta - \lfloor x \pm \beta \rfloor)$.

To prove \eqref{e:quadboundary}, let $x \in \mathbb{Z}[\beta^{-1}] \cap [0,1)$, and set $y = x + \beta - \lfloor x + \beta \rfloor$, $z = x - \beta - \lfloor x - \beta \rfloor$, $\varepsilon = \min\{\widehat{x}-x, \widehat{y}-y, \widehat{z} - z\}$. 
For each $u \in \mathbb{Z}[\beta] \cap [0,\varepsilon)$, we have
\begin{equation} \label{e:xu}
\mathcal{R}(x) \cap \mathbb{K}'_\infty \times \delta_\mathrm{f}(\{-u\}) = \big(\mathcal{S}(x+u) - \delta'_\infty(u)\big) \times \delta_\mathrm{f}(\{-u\}),
\end{equation}
and $\mathcal{S}(x+u)$ is an interval by Theorem~\ref{th:int}~(\ref{i:t35}).
Since $\overline{\delta_\mathrm{f}(\mathbb{Z}[\beta] \cap [0,\varepsilon))} = \overline{\delta_\mathrm{f}(\mathbb{Z}[\beta])}$ and $\mathcal{R}(x)$ is the closure of its interior, each $\mathbf{z} \in \partial \mathcal{R}(x)$ can be approximated by endpoints of intervals $(\mathcal{S}(x+u) - \delta'_\infty(u)) \times \delta_\mathrm{f}(\{-u\})$, $u \in \mathbb{Z}[\beta] \cap [0,\varepsilon)$.
By the proof of Theorem~\ref{th:int}~(\ref{i:t35}) in Section~\ref{sec:quadr-pisot-numb} and since $y+u < \widehat{y}$, $z+u < \widehat{z}$, each endpoint of $\mathcal{S}(x+u)$ lies also in $\mathcal{S}(y+u)$ or $\mathcal{S}(z+u)$.
As \eqref{e:xu} still holds if we replace $x$ by $y$ or~$z$, we obtain that $\mathbf{z} \in \mathcal{R}(y) \cup \mathcal{R}(z)$.

\subsection{Pruned boundary graph}
Formula~\eqref{e:quadboundary} suggests the following definition: The \emph{pruned boundary graph} of a quadratic Pisot number~$\beta$ is the subgraph of the boundary graph that is induced by the set of nodes $[v,x,w]$ with $x \in \pm\{\beta-\lfloor\beta\rfloor, \lceil\beta\rceil-\beta\}$.

By~\eqref{e:quadboundary}, we can replace the boundary graph by its pruned version in \eqref{e:xyintersection} and~\eqref{e:xyintersection2}.
This allows us to simplify some arguments of \cite[Section~5]{Akiyama-Barat-Berthe-Siegel:08}.

\medskip
In the following, let $\beta^2 = a\beta + b$ with $a \ge b \ge 1$.
Then we have $V = \{0, v\}$ with $v = \beta - a$.
We do not consider the case of negative~$b$ because we know from \cite[Proposition~5]{Akiyama:98} that $\mathrm{Pur}(\beta) \cap \mathbb{Q} = \{0\}$ when $\beta$ has a positive real conjugate, which implies that $\gamma(\beta) = 0$. 

For the description of the pruned boundary graph, we have to distinguish two cases:
\begin{itemize}
\item
If $2b \le a$, then the transitions of the pruned boundary graph are
\begin{align*}
[v,{v\!-\!1},0] \stackrel{(d,d+a-b+1)}\longrightarrow [0,{1\!-\!v},v] & \qquad  (0 \le d < b), \\
[0,{1\!-\!v},v] \stackrel{(d,d-a+b-1)}\longrightarrow [v,{v\!-\!1},0] & \qquad (a-b < d \le a), \\
[0,v,v], [v,v,v] \stackrel{(d,d+a-b)}\longrightarrow [0,{1\!-\!v},v] & \qquad (0 \le d < b), \\
[v,-v,0], [v,-v,v] \stackrel{(d,d-a+b)}\longrightarrow [v,{v\!-\!1},0] & \qquad (a-b \le d < a).
\end{align*}
\item
If $2b > a$, then the transitions of the pruned boundary graph are
\begin{align*}
[v,{v\!-\!1},0], [0,{v\!-\!1},0] \stackrel{(d,d+a-b+1)}\longrightarrow [0,{1\!-\!v},0] & \qquad  (0 \le d \le 2b-a-2), \\
[v,{v\!-\!1},0], [0,{v\!-\!1},0] \stackrel{(d,d+a-b+1)}\longrightarrow [0,{1\!-\!v},v] & \qquad  (2b-a-1 \le d < b), \\
[0,{1\!-\!v},v], [0,{1\!-\!v},0] \stackrel{(d,d-a+b-1)}\longrightarrow [0,{v\!-\!1},0] & \qquad (a-b < d < b), \\
[0,{1\!-\!v},v], [0,{1\!-\!v},0] \stackrel{(d,d-a+b-1)}\longrightarrow [v,{v\!-\!1},0] & \qquad (b \le d \le a), \\
[0,v,v] \stackrel{(d,d+a-b)}\longrightarrow [0,{1\!-\!v},0] & \qquad (0 \le d < 2b-a), \\
[0,v,v] \stackrel{(d,d+a-b)}\longrightarrow [0,{1\!-\!v},v] & \qquad (2b-a \le d < b), \\
[v,-v,0] \stackrel{(d,d-a+b)}\longrightarrow [0,{v\!-\!1},0] & \qquad (a-b \le d < b), \\
[v,-v,0] \stackrel{(d,d-a+b)}\longrightarrow [v,{v\!-\!1},0] & \qquad (b \le d < a).
\end{align*}
\end{itemize}
To prove that these are exactly the transitions of the pruned boundary graph, note first that the states of the graph are of the form $[u,x,w]$ with $u,w \in \{0,v\}$ and $x \in \pm\{v, {1\!-\!v}\}$. 
The only possibilites are thus $[0,v,v]$, $[0,{1\!-\!v},v]$, $[0,{1\!-\!v},0]$ if $2v > 1$, $[v,v,v]$ if $2v < 1$, and their negatives $[v,-v,0]$, $[v,{v\!-\!1},0]$, $[0,{v\!-\!1},0]$, and $[v,-v,v]$ respectively.
Since $v = \frac{b}{\beta}$, we have $2v > 1$ if and only if $2b > a$. 
Moreover, the only possibility for $\frac{v+d}{\beta} \in \pm\{v, {1\!-\!v}\}$ with $d \le a$ is that $\frac{v+d}{\beta} = 1+\frac{d-a}{\beta} = 1-v$, i.e., $d = a-b$. 
Therefore, the above lists contain all possible transitions.
Since we have an infinite path starting from each node, all given nodes correspond to the intersection of some tiles and are thus in the boundary graph.

\medskip
From the description of the boundary graph, we see that the paths starting in a state $[u,x,w]$ depend only on~$x$. 
Therefore, we can merge the states with same middle component and obtain the graph in Figure~\ref{f:prunedgraph}.
Note that we have exactly $|N(\beta)|$ outgoing transitions from each state, which can be explained by the fact that the intersection $\mathcal{R}(x) \cap \mathcal{R}(x+\beta-\lfloor x+\beta\rfloor) \cap \mathbb{R} \times \{y\}$ consists of a singleton for each $x \in \mathbb{Z}[\beta]$, $y \in \overline{\delta_\mathrm{f}(\mathbb{Z}[\beta])}$. 

\begin{figure}[ht]
\begin{tikzpicture} [->,>=stealth',shorten >=1pt,auto,node distance=4cm,thick,main node/.style={circle,fill=blue!10,draw}] 
\node[main node] (1) {$\beta-a$};
\node[main node] (2) [right of=1] {$\!a\!+\!1\!-\!\beta\!$};
\node[main node] (3) [right of=2] {$\!\beta\!-\!a\!-\!1\!$};
\node[main node] (4) [right of=3] {$a-\beta$};
\path[every node/.style={font=\sffamily\small}]
(1) edge node [above] {$\begin{array}{c}(0,a\!-\!b)\\\vdots\\(b\!-\!1,a\!-\!1)\end{array}$} (2)
(2) edge [bend left] node[above] {$\begin{array}{c}(a\!-\!b\!+\!1,0)\\\vdots\\(a,b\!-\!1)\end{array}$} (3)
(3) edge [bend left] node[below] {$\begin{array}{c}(0,a\!-\!b\!+\!1)\\\vdots\\(b\!-\!1,a)\end{array}$} (2)
(4) edge node [above] {$\begin{array}{c}(a\!-\!b,0)\\\vdots\\(a\!-\!1,b\!-\!1)\end{array}$} (3);
\end{tikzpicture}
\caption{The pruned boundary graph (after merging the states with same middle component) for $\beta^2=a\beta+b$, $a \ge b \ge 1$. \label{f:prunedgraph}}
\end{figure}
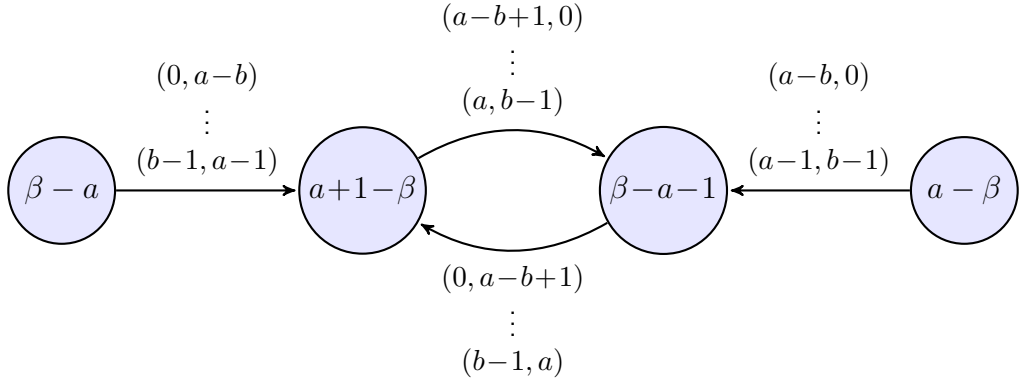

\subsection{Proof of Theorem~\ref{t:quadratic}}
Let $\beta^2 = a \beta + b$, $a \ge b \ge 1$, and $\beta' = -b\beta^{-1} = a-\beta$ be the Galois conjugate of~$\beta$.
By Theorem~\ref{t:gammabeta}, we have
\begin{align}\label{e:inf}
\gamma(\beta) & \ge \inf \big(\{1\} \cup \big\{x \in \mathbb{Q} \cap [0,\beta\!-\!a):\, \delta'_\infty(-x) \in \pi'_\infty\big(Z'\setminus\mathcal{R}(0)\big)\big\} \\ \nonumber
& \qquad \quad \cup \big\{x \in \mathbb{Q} \cap [\beta\!-\!a,1):\, \delta'_\infty(-x) \in \pi'_\infty\big(Z'\setminus\mathcal{R}(\beta\!-\!a)\big) - \delta'_\infty(\beta\!-\!a)\big\} \big),
\end{align}
with equality if $\overline{\delta_\mathrm{f}(\mathbb{Q})} = \mathbb{K}_\mathrm{f}$. 
By Lemma~\ref{l:qdense}, the latter equality holds if $\mathrm{gcd}(a,b) = 1$.
We have to show that the infimum is equal to $\max\{0, 1-\frac{(b-1)b\,\beta}{\beta^2-b^2}\}$. 

By~\eqref{e:quadboundary} and its proof, $\pi'_\infty(Z'\setminus\mathcal{R}(0))$ is the union of two half-lines:
\[
\pi'_\infty\big(Z'\setminus\mathcal{R}(0)\big) = \Big(-\infty, \max \pi'_\infty\big(\mathcal{R}(0) \cap \mathcal{R}({\beta\!-\!a})\big)\Big) \cup \Big(\min \pi'_\infty\big(\mathcal{R}(0) \cap \mathcal{R}({a\!+\!1\!-\!\beta}), \infty\big)\Big).
\]
The point in $\mathcal{R}(0) \cap \mathcal{R}({\beta\!-\!a})$ that realises $\max \pi'_\infty(\mathcal{R}(0) \cap \mathcal{R}({\beta\!-\!a}))$ is given by the following infinite walk in the pruned boundary graph starting from ${\beta\!-\!a}$: choose the transition to ${a\!+\!1\!-\!\beta}$ with maximal first digit ${b\!-\!1}$, then the transition to ${\beta\!-\!a\!-\!1}$ with minimal first digit ${a\!-\!b\!+\!1}$ (since we multiply it by an odd power of $\beta'<0$), again the transition to ${a\!+\!1\!-\!\beta}$ with maximal first digit ${\!b\!-\!1}$, etc.
This gives
\begin{align*}
\max \pi'_\infty\big(\mathcal{R}(0) \cap \mathcal{R}({\beta\!-\!a})\big) & = \sum_{j=0}^\infty \big(b-1 + (a-b+1)\, \beta')\big)\, (\beta')^{2j} \\
& = \frac{b-1+(a-b+1)\,\beta'}{1-(\beta')^2} = \frac{(1-b)\,\beta'}{1-(\beta')^2} - 1 = \frac{(b-1)b\,\beta}{\beta^2-b^2} - 1,
\end{align*}
Note that $\delta'_\infty(x) = x$ for $x \in \mathbb{Q}$. 
If $\frac{(b-1)b\,\beta}{\beta^2-b^2} \ge 1$, then we obtain that
\[
\inf \{x \in \mathbb{Q} \cap [0,\beta\!-\!a):\, \delta'_\infty(-x) \in \pi'_\infty(Z'\setminus\mathcal{R}(0))\} = 0 = \max\big\{0, 1-\tfrac{(b-1)b\,\beta}{\beta^2-b^2}\big\}.
\]

Assume now that $\frac{(b-1)b\,\beta}{\beta^2-b^2} < 1$.
Then, by similar calculations as above, we obtain that
\[
\min \pi'_\infty\big(\mathcal{R}(0) \cap \mathcal{R}({a\!+\!1\!-\!\beta})\big) = \frac{a-b+1+(b-1)\,\beta'}{1-(\beta')^2} = \beta\, \bigg(1 - \frac{(b-1)\beta}{\beta^2-b^2} \bigg) > 0.
\]
Therefore, we have $\big[\frac{(b-1)b\,\beta}{\beta^2-b^2}-1,0\big] \times \overline{\delta_\mathrm{f}(\mathbb{Z}[\beta])} \subseteq \mathcal{R}(0)$, and 
\[
\inf \big\{x \in \mathbb{Q} \cap [0,1):\, \delta'_\infty(-x) \in \pi'_\infty\big(Z'\setminus\mathcal{R}(0)\big)\big\} = 1 - \tfrac{(b-1)b\,\beta}{\beta^2-b^2} = \max\big\{0, 1-\tfrac{(b-1)b\,\beta}{\beta^2-b^2}\big\}.
\]
This concludes the case $0 < 1-\tfrac{(b-1)b\,\beta}{\beta^2-b^2} < \beta-a$. 

Assume now that $\frac{(b-1)b\,\beta}{\beta^2-b^2} \le a+1-\beta$.
We see from the boundary graph that
\[
\max \pi'_\infty\big(\mathcal{R}({\beta\!-\!a}) \cap \mathcal{R}({2\beta\!-\!\lfloor2\beta\rfloor})\big) - \delta'_\infty(\beta\!-\!a) = \max \pi'_\infty\big(\mathcal{R}(0) \cap \mathcal{R}({\beta\!-\!a})\big)
\]
and, since the smallest first digit in the outgoing transitions from $a\!-\!\beta$ is $a\!-\!b$,
\begin{align*}
& \min \pi'_\infty(\mathcal{R}({\beta\!-\!a}) \cap \mathcal{R}(0)) - \delta'_\infty(\beta\!-\!a) = \min \pi'_\infty\big(\mathcal{R}(0) \cap \mathcal{R}({a\!+\!1\!-\!\beta})\big) - 1 \\
& \qquad = \beta - \frac{(b-1)\beta^2}{\beta^2-b^2} - 1 \ge \beta - \frac{(a+1-\beta)\,\beta}{b} -1 = \beta\, \bigg(1-\frac{1}{b}\bigg) \ge 0.
\end{align*}
Similarly as above, we obtain that
\[
\inf \big\{x \in \mathbb{Q} \cap [\beta\!-\!a,1):\, \delta'_\infty(-x) \in \pi'_\infty\big(Z'\setminus\mathcal{R}(\beta\!-\!a)\big) - \delta'_\infty(\beta\!-\!a)\big\} = 1 - \tfrac{(b-1)b\,\beta}{\beta^2-b^2},
\]
provided that $b \ge 2$.
Finally, for $b=1$, we obtain Schmidt's equality $\gamma(\beta) = 1$ \cite{Schmidt:80}. 

To conclude the proof of Theorem~\ref{t:quadratic}, we show that $\frac{(b-1)b\,\beta}{\beta^2-b^2} < 1$ if and only if $(b-1)b < a$.
Indeed, if $(b-1)b \ge a$, then 
\[
\beta^2 - b^2 - (b-1)b\,\beta \le \beta^2 - a\beta - b^2 = b - b^2 \le -a < 0,
\] 
and, if $(b-1)b < a$, then 
\[
\beta^2 - b^2 - (b-1)b\,\beta \ge \beta^2 - (a-1)\,\beta - (a+b) = \beta - a > 0.
\] 

\subsection{Example of $\gamma(\beta)$.} 
Let $\beta = \frac{3+\sqrt{17}}{2}$, i.e., $\beta^2=3\beta+2$. Then, by Theorem~\ref{t:quadratic},
\[  
\gamma(\beta) = 1 - \frac{2\,\beta}{\beta^2-4} = \frac{1}{\beta+2} \approx  0,1798.
\]
Figure~\ref{fig:gamma} illustrates how to obtain $\gamma(\beta)$ as $\min \pi'_\infty\big(-\mathcal{R}(0) \cap -\mathcal{R}({\beta\!-\!3})\big)$.

\begin{figure}[ht]
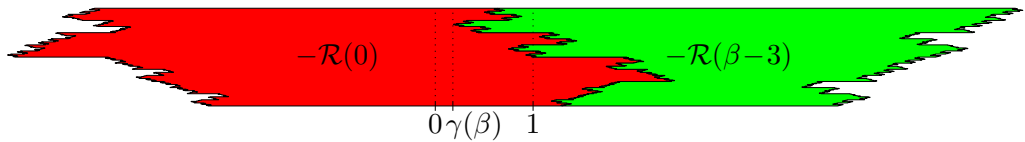


\caption{Visualization of $\gamma(\beta)$ for $\beta^2=3\beta+2$.\label{fig:gamma}}
\end{figure}

\bibliographystyle{amsalpha}
\bibliography{nonunit}
\end{document}